\documentclass[a4paper, oneside]{article}
\usepackage{enumitem}
\usepackage{lmodern}
\usepackage{amssymb,amsthm,amsmath}
\usepackage[utf8]{inputenc}
\usepackage{verbatim}
\usepackage{tikz}
\usepackage{authblk}
\usepackage{cite}
\usepackage{hyperref}
\usepackage{aligned-overset} 

\usetikzlibrary{math} 
\usetikzlibrary{patterns} 
\usetikzlibrary{arrows}
\usetikzlibrary{decorations.pathreplacing}
\usetikzlibrary{decorations.pathmorphing}
\usetikzlibrary{automata}
\usetikzlibrary{positioning}

\theoremstyle{definition}
\newtheorem{definition}{Definition}[section]
\newtheorem{lemma}[definition]{Lemma}
\newtheorem{proposition}[definition]{Proposition}
\newtheorem{theorem}[definition]{Theorem}
\newtheorem{example}[definition]{Example}
\newtheorem{corollary}[definition]{Corollary}
\newtheorem{remark}[definition]{Remark}

\newcommand{\N}{\mathbb{N}}
\newcommand{\Zpos}{\mathbb{Z}_+}
\newcommand{\Z}{\mathbb{Z}}
\newcommand{\Q}{\mathbb{Q}}
\newcommand{\R}{\mathbb{R}}

\newcommand{\digs}{\Sigma} 
\newcommand{\alp}{\digs} 
\newcommand{\Mul}{\Pi} 
\newcommand{\mul}{g} 
\newcommand{\trsh}{\Xi} 
\newcommand{\num}[1]{{\mathcal{N}(#1)}} 
\newcommand{\fin}[1]{{\mathcal{F}(#1)}} 
\newcommand{\cvec}[1]{{\mathbf{#1}}} 
\newcommand{\dedge}{{D}} 
\newcommand{\edge}{{E}} 

\DeclareMathOperator{\config}{config} 
\DeclareMathOperator{\real}{real} 
\DeclareMathOperator{\tr}{Tr} 
\DeclareMathOperator{\base}{base} 
\DeclareMathOperator{\cube}{cube} 
\DeclareMathOperator{\val}{val} 
\DeclareMathOperator{\diag}{diag} 
\DeclareMathOperator{\tess}{tess} 
\DeclareMathOperator{\emb}{{\iota}} 
\DeclareMathOperator{\Emb}{I} 
\DeclareMathOperator{\topi}{\tau} 
\DeclareMathOperator{\boti}{\beta} 
\DeclareMathOperator{\faceat}{faceAt} 
\DeclareMathOperator{\ins}{\vee} 
\DeclareMathOperator{\insover}{\not\vee} 
\DeclareMathOperator{\wgt}{wgt} 
\DeclareMathOperator{\abel}{Ab}
\DeclareMathOperator{\lbl}{\lambda} 
\DeclareMathOperator{\macro}{M} 
\DeclareMathOperator{\micro}{\mu} 
\DeclareMathOperator{\conj}{Conj} 
\DeclareMathOperator{\fact}{Fact} 
\DeclareMathOperator{\sign}{sign} 
\DeclareMathOperator{\block}{B} 

\DeclareMathOperator{\fractional}{frac} 
\DeclareMathOperator{\integ}{int} 

\def\tilescale[#1,#2,#3,#4](#5,#6){\tikzmath{int \n,\w,\s,\e; real \scale; \n=Mod(#1,#2); \w=#1/#2; \s=#1/#3; \e=Mod(#1,#3); \scale=#4;} \draw (#5,#6) -- ++(-\scale,0) node[midway,yshift=3pt,anchor=north]{\footnotesize{\n}} -- ++(0,-\scale) node[midway,xshift=-3pt,anchor=west]{\footnotesize{\w}} -- ++(\scale,0) node[midway,yshift=-3pt,anchor=south]{\footnotesize{\s}}-- ++(0,\scale) node[midway,xshift=3pt,anchor=east]{\footnotesize{\e}} ++(-\scale/2,-\scale/2) node{\large #1}}

\def\tile[#1,#2,#3](#4,#5){\tilescale[#1,#2,#3,1](#4,#5)}

\def\gridscale[#1,#2,#3,#4](#5,#6){\tikzmath{int \n,\w,\s,\e; real \scale; \n=Mod(#1,#2); \w=#1/#2; \s=#1/#3; \e=Mod(#1,#3); \scale=#4;}
\node(ne) at (#5,#6) [shape=circle,draw,minimum size=6mm] {};
\node(nw) at (#5-\scale,#6) [shape=circle,draw,minimum size=6mm] {};
\node(sw) at (#5-\scale,#6-\scale) [shape=circle,draw,minimum size=6mm] {};
\node(se) at (#5,#6-\scale) [shape=circle,draw,minimum size=6mm] {};
\draw[-{triangle 45}] (nw) to [below] node {\footnotesize{\n}} (ne);
\draw[-{triangle 45}] (sw) to [right] node {\footnotesize{\w}} (nw);
\draw[-{triangle 45}] (sw) to [above] node {\footnotesize{\s}} (se);
\draw[-{triangle 45}] (se) to [left] node {\footnotesize{\e}} (ne)
}

\def\gridscalewt[#1,#2,#3,#4,#5](#6,#7){\tikzmath{int \n,\w,\s,\e; int \ne,\nw,\sw,\se; real \scale; \n=Mod(#1,#2); \w=#1/#2; \s=#1/#3; \e=Mod(#1,#3); \ne=#5; \nw=#5*#2; \sw=#5*#2*#3; \se=#5*#3;\scale=#4;}
\node(ne) at (#6,#7) [shape=circle,draw,minimum size=6mm] {\ne};
\node(nw) at (#6-\scale,#7) [shape=circle,draw,minimum size=6mm] {\nw};
\node(sw) at (#6-\scale,#7-\scale) [shape=circle,draw,minimum size=6mm] {\sw};
\node(se) at (#6,#7-\scale) [shape=circle,draw,minimum size=6mm] {\se};
\draw[-{triangle 45}] (nw) to [below] node {\footnotesize{\n}} (ne);
\draw[-{triangle 45}] (sw) to [right] node {\footnotesize{\w}} (nw);
\draw[-{triangle 45}] (sw) to [above] node {\footnotesize{\s}} (se);
\draw[-{triangle 45}] (se) to [left] node {\footnotesize{\e}} (ne)
 }

\def\toptiltscale[#1,#2,#3,#4](#5,#6){\tikzmath{int \n,\w,\s,\e; real \scale; \n=Mod(#1,#2); \w=#1/#2; \s=#1/#3; \e=Mod(#1,#3); \scale=#4;} \draw (#5,#6) -- ++(-3*\scale,0) node[midway,yshift=-4pt]{\footnotesize{\n}} -- ++(-\scale,-\scale) node[midway,xshift=5pt]{\footnotesize{\w}} -- ++(3*\scale,0) node[midway,yshift=4pt]{\footnotesize{\s}}-- ++(\scale,\scale) node[midway,xshift=-5pt]{\footnotesize{\e}} ++(-2*\scale,-\scale/2) node{\large #1}}

\def\righttiltscale[#1,#2,#3,#4](#5,#6){\tikzmath{int \n,\w,\s,\e; real \scale; \n=Mod(#1,#2); \w=#1/#2; \s=#1/#3; \e=Mod(#1,#3); \scale=#4;} \draw (#5,#6) -- ++(-\scale,-\scale) node[midway,yshift=-5pt]{\footnotesize{\n}} -- ++(0,-3*\scale) node[midway,xshift=3pt]{\footnotesize{\w}} -- ++(\scale,\scale) node[midway,yshift=5pt]{\footnotesize{\s}}-- ++(0,3*\scale) node[midway,xshift=-3pt]{\footnotesize{\e}} ++(-\scale/2,-2*\scale) node{\large #1}}

\def\cubefigscale[#1,#2,#3,#4,#5](#6,#7){\tikzmath{int \x,\y,\z; real \scale; \z=Mod(#1,#2*#3); \x=Mod(#1,#3*#4); \y=#1/#3; \scale=#5;} \toptiltscale[\z,#2,#3,\scale](#6,#7); \righttiltscale[\x,#3,#4,\scale](#6,#7); \tilescale[\y,#2,#4,3*\scale](#6-\scale/3,#7-\scale/3)}

\def\cubefig[#1,#2,#3,#4](#5,#6){\cubefigscale[#1,#2,#3,#4,1](#5,#6)}

\def\topgridscale[#1,#2,#3,#4](#5,#6){\tikzmath{int \n,\w,\s,\e; real \scale; \n=Mod(#1,#2); \w=#1/#2; \s=#1/#3; \e=Mod(#1,#3); \scale=#4;}
\node(ne) at (#5,#6) [shape=circle,draw,minimum size=6mm] {};
\node(nw) at (#5-3*\scale,#6) [shape=circle,draw,minimum size=6mm] {};
\node(sw) at (#5-4*\scale,#6-\scale) [shape=circle,draw,minimum size=6mm] {};
\node(se) at (#5-\scale,#6-\scale) [shape=circle,draw,minimum size=6mm] {};
\draw[-{triangle 45}] (nw) to [below] node {\footnotesize{\n}} (ne);
\draw[-{triangle 45}] (sw) to [right] node {\footnotesize{\w}} (nw);
\draw[-{triangle 45}] (sw) to [above] node {\footnotesize{\s}} (se);
\draw[-{triangle 45}] (se) to [left] node {\footnotesize{\e}} (ne)
 }

\def\rightgridscale[#1,#2,#3,#4](#5,#6){\tikzmath{int \n,\w,\s,\e; real \scale; \n=Mod(#1,#2); \w=#1/#2; \s=#1/#3; \e=Mod(#1,#3); \scale=#4;}
\node(ne) at (#5,#6) [shape=circle,draw,minimum size=6mm] {};
\node(nw) at (#5-\scale,#6-\scale) [shape=circle,draw,minimum size=6mm] {};
\node(sw) at (#5-\scale,#6-4*\scale) [shape=circle,draw,minimum size=6mm] {};
\node(se) at (#5,#6-3*\scale) [shape=circle,draw,minimum size=6mm] {};
\draw[-{triangle 45}] (nw) to [below] node {\footnotesize{\n}} (ne);
\draw[-{triangle 45}] (sw) to [right] node {\footnotesize{\w}} (nw);
\draw[-{triangle 45}] (sw) to [above] node {\footnotesize{\s}} (se);
\draw[-{triangle 45}] (se) to [left] node {\footnotesize{\e}} (ne)
 }

\def\cubegridscale[#1,#2,#3,#4,#5](#6,#7){\tikzmath{int \x,\y,\z; real \scale; \z=Mod(#1,#2*#3); \x=Mod(#1,#3*#4); \y=#1/#3; \scale=#5;} \topgridscale[\z,#2,#3,\scale](#6,#7); \rightgridscale[\x,#3,#4,\scale](#6,#7); \gridscale[\y,#2,#4,3*\scale](#6-\scale/3,#7-\scale/3)}

\def\cubegrid[#1,#2,#3,#4](#5,#6){\cubegridscale[#1,#2,#3,#4,1](#5,#6)}

\def\topgridscalewt[#1,#2,#3,#4,#5](#6,#7){\tikzmath{int \n,\w,\s,\e; int \ne,\nw,\sw,\se; real \scale; \n=Mod(#1,#2); \w=#1/#2; \s=#1/#3; \e=Mod(#1,#3); \ne=#5; \nw=#5*#2; \sw=#5*#2*#3; \se=#5*#3;\scale=#4;}
\node(ne) at (#6,#7) [shape=circle,draw,minimum size=6mm] {\ne};
\node(nw) at (#6-3*\scale,#7) [shape=circle,draw,minimum size=6mm] {\nw};
\node(sw) at (#6-4*\scale,#7-\scale) [shape=circle,draw,minimum size=6mm] {\sw};
\node(se) at (#6-\scale,#7-\scale) [shape=circle,draw,minimum size=6mm] {\se};
\draw[-{triangle 45}] (nw) to [below] node {\footnotesize{\n}} (ne);
\draw[-{triangle 45}] (sw) to [right] node {\footnotesize{\w}} (nw);
\draw[-{triangle 45}] (sw) to [above] node {\footnotesize{\s}} (se);
\draw[-{triangle 45}] (se) to [left] node {\footnotesize{\e}} (ne)
 }

\def\rightgridscalewt[#1,#2,#3,#4,#5](#6,#7){\tikzmath{int \n,\w,\s,\e; int \ne,\nw,\sw,\se; real \scale; \n=Mod(#1,#2); \w=#1/#2; \s=#1/#3; \e=Mod(#1,#3); \ne=#5; \nw=#5*#2; \sw=#5*#2*#3; \se=#5*#3;\scale=#4;}
\node(ne) at (#6,#7) [shape=circle,draw,minimum size=6mm] {\ne};
\node(nw) at (#6-\scale,#7-\scale) [shape=circle,draw,minimum size=6mm] {\nw};
\node(sw) at (#6-\scale,#7-4*\scale) [shape=circle,draw,minimum size=6mm] {\sw};
\node(se) at (#6,#7-3*\scale) [shape=circle,draw,minimum size=6mm] {\se};
\draw[-{triangle 45}] (nw) to [below] node {\footnotesize{\n}} (ne);
\draw[-{triangle 45}] (sw) to [right] node {\footnotesize{\w}} (nw);
\draw[-{triangle 45}] (sw) to [above] node {\footnotesize{\s}} (se);
\draw[-{triangle 45}] (se) to [left] node {\footnotesize{\e}} (ne)
 }

\def\cubegridscalewt[#1,#2,#3,#4,#5,#6](#7,#8){\tikzmath{int \x,\y,\z; int \newwgt real \scale; \z=Mod(#1,#2*#3); \x=Mod(#1,#3*#4); \y=#1/#3; \scale=#5; \newwgt=#6*#3;} 
\topgridscalewt[\z,#2,#3,\scale,#6](#7,#8); 
\rightgridscalewt[\x,#3,#4,\scale,#6](#7,#8);
\gridscalewt[\y,#2,#4,3*\scale,\newwgt](#7-\scale/3,#8-\scale/3)}

\def\cubegridwt[#1,#2,#3,#4,#5](#6,#7){\cubegridscalewt[#1,#2,#3,#4,1,#5](#6,#7)}

\begin{document}

\title{Multiplication cubes and multiplication automata}

\author{Johan Kopra}

\affil{Department of Mathematics and Statistics, \\FI-20014 University of Turku, Finland}
\affil{jtjkop@utu.fi}

\date{}

\maketitle

\setcounter{page}{1}

\begin{abstract}
\noindent We extend previously known two-dimensional multiplication tiling systems that simulate multiplication by two natural numbers $p$ and $q$ in base $pq$ to higher dimensional multiplication tessellation systems. We develop the theory of these systems and link different multiplication tessellation systems with each other via macrotile operations that glue cubes in one tessellation system into larger cubes of another tessellation system. The macrotile operations yield topological conjugacies and factor maps between cellular automata performing multiplication by positive numbers in various bases.
\end{abstract}

\providecommand{\keywords}[1]{\textbf{Keywords:} #1}
\noindent\keywords{symbolic dynamics, Wang tiles, multiplication cubes, cellular automata, multiplication automata}

\section{Introduction}

Cellular automata (CA) are symbolic dynamical systems that act on either $\alp^\N$ (one-sided CA) or $\alp^\Z$ (two-sided CA) for some finite symbol set $\alp$ and that are defined by using a local rule. The CA appearing in this paper are multiplication automata $\Mul_{\alpha,N}$ that simulate multiplication by some $\alpha>0$ on base-$N$ representations of numbers. These are not defined for all pairs of $\alpha$ and $N$. In~\cite{BHM96}, possibly the first paper featuring such CA, it is shown that in the case of one-sided multiplication automata (meaning that multiplication is performed only on the fractional part of a number) $\Mul_{\alpha,N}$ is defined precisely when $\alpha$ is an integer all of whose prime factors divide $N$. The paper~\cite{BM97} shows for $\alpha$ dividing $N$ that $\Mul_{\alpha,N}$ is topologically conjugate to the left shift map on $\digs_\alpha^\N$ with  $\digs_\alpha=\{0,1,\dots,\alpha-1\}$ if and only $\alpha$ and $N$ are divisible by the same prime numbers. This left shift map is in fact equal to $\Mul_{\alpha,\alpha}$, which is the simplest way to represent the operation of multiplying by $\alpha$.  Considering two-sided multiplication automata makes it possible to also talk about automata multiplying by proper fractions, as was done in~\cite{Kari12a} in the special case $\alpha=3/2$, $N=6$. From~\cite{BHM96} one can infer that in the two-sided case $\Mul_{\alpha,N}$ is defined precisely when $\alpha$ is a rational number whose numerator and denominator are products of factors of $N$.

In this paper we consider explicitly only the case of two-sided multiplication automata and look further into topological conjugacies of multiplication automata with other multiplication automata that are not necessarily shift maps as was the case in~\cite{BM97}. We show in Corollary~\ref{mulConj} that multiplication automata $\Mul_{\alpha,N_1}$ and $\Mul_{\alpha,N_2}$ are conjugate when $N_1$ and $N_2$ are divisible by the same prime numbers. Even when this condition is not satisfied, we can consider topological factors. We show in Corollary~\ref{mulFact} that $\Mul_{\alpha,N_1}$ has $\Mul_{\alpha,N_2}$ as a factor if every prime factor of $N_2$ is a prime factor of $N_1$. The conjugacies and factor maps we present for connecting these multiplication automata essentially transform base-$N_1$ representations of any nonnegative real number to base-$N_2$ representations of the same real number.

To prove the results mentioned in the previous paragraph, this paper takes the approach of considering so-called multiplication cubes and multiplication tessellations, which are interesting mathematical objects in their own right. In~\cite{Rud90} multiplication by $p$ and $q$ is essentially implemented by tilings of the plane with the tiles being base-$pq$ digits and multiplication by $p$ and $q$ corresponding to shifting the tiling either horizontally or vertically. The tiling terminology appears explicitly in~\cite{CSW21} in the special case $p=2$, $q=3$. Instead of restricting to two-dimensional tilings, we will introduce multiplication tessellation systems using multiplication (hyper)cubes of arbitrary dimension.

Our construction of multiplication cubes is based on mixed base representations of natural numbers, the basics of which we cover in Section~\ref{mixedBaseSect}. Section~\ref{WangSect} is all about multiplication cubes, and mixed base representations of numbers are used to define multiplication cubes and their tessellations in Subsection~\ref{cubePre}. Briefly put, associated to any vector of positive integers $(n_1,n_2,\dots,n_d)$ there will be $N=n_1n_2\cdots n_d$ different $d$-dimensional hypercubes correspoding to the base-$N$ digits and having various labels on their lower dimensional hyperfaces, and in valid $d$-dimensional tessellations neighboring cubes have to have matching labels on adjacent $(d-1)$-dimensional hyperfaces. The underlying idea is that then one can associate any real number $\xi\geq0$ with a valid tessellation containing a base-$N$ representation of $\xi$ along the main diagonal (Corollary~\ref{tessExist}), and moving along the direction of the $i$th coordinate axis in a tessellation corresponds to multiplication by $n_i$ (Proposition~\ref{infShiftMul}).

We proceed to investigate various structures appearing in multiplication tilings in later subsections of Section~\ref{WangSect}. In Theorem~\ref{subfaceNeigh} of Subsection~\ref{lowerMatchSubSect} it turns out that the matching condition of $(d-1)$-dimensional hyperfaces in valid tessellations automatically implies that also the labels of lower dimensional hyperfaces in adjacent cubes match. In Subsection~\ref{pathSubSect} we define the notion of a path integral over a path in a tessellation. This terminology is justified by Theorem~\ref{indPath}, according to which the path integral over a cycle in a valid tessellation is always equal to zero. This is then used as a technical tool to define labels between directed line segments connecting any two points in a tessellation. These labels are used in Subsection~\ref{macMicSubSect}, which introduces the macrotile operation. The macrotile operation groups multiplication cubes into larger cubes that can also be viewed as multiplication cubes. More generally, with some restrictions it is possible to draw a grid of parallelepipeds within a tessellation and interpret each individual parallelelepiped as a new multiplication cube. The way this works is that the labels of the edges of the parallelepipeds in the original tessellation yield the labels of the edges of the new multiplication cubes. The grid of parallelepipeds may be of a lower dimension than the original tessellation, which in particular means that a lower dimensional cut along the $i_1$th, $i_2$th,$\dots$,$i_k$th coordinate axes in a $d$-dimensional tessellation is a $k$-dimensional tessellation over the multiplication cube set associated to the vector $(n_{i_1},n_{i_2},\dots,n_{i_k})$. By Theorem~\ref{macromicroInv} the macrotile operation has in some cases an inverse map, the microtile operation. Even when the macrotile operation is not invertible, by Theorem~\ref{macroSurj} it is always a surjection between sets of valid tessellations.

In Section~\ref{CASect} we turn to multiplication automata, which are defined in Subsection~\ref{CAPre}. In Subsection~\ref{WangCASubSect} multiplication automata are connected to multiplication tessellations: by Theorem~\ref{tessdiag} multiplication automata are topologically conjugate to shift maps on sets of valid tessellations by multiplication cubes. Due to this connection, conjugacy and factor relations between various shift dynamics on multiplication tessellations imply the conjugacy and factor relations between multiplication automata mentioned earlier in this introduction. As a minor application we use these conjugacy and factor relations in Subsection~\ref{regSubSect}, together with some earlier results, to completely classify multiplication automata according to their regularity status in the sense of~\cite{Kur97}.

\section{Preliminaries}

We denote the set of positive integers by $\Zpos$ and define the set of natural numbers by $\N=\Zpos\cup\{0\}$. For $n\in\Zpos$ we denote $\digs_n=\{0,1,\dots,n-1\}$. Whenever $A$ and $B$ are sets, $A^B$ denotes the collection of functions from $B$ to $A$. We often denote the value of a function $f\in A^B$ at $b\in B$ by $f[b]$ instead of $f(b)$. When $B$ is countable and $A$ is finite, the set $A^B$ is a compact metrizable space with respect to the prodiscrete topology, and every closed subset of $A^B$ is also compact and metrizable with respect to the subspace topology.

We interpret the notation $A^n$ for a set $A$ and $n\in\N$ as a shorthand for $A^{\{1,2,\dots,n\}}$, the set of sequences of length $n$ over $A$ indexed by $1,2,\dots,n$. The elements of this set can be represented by $(a_1,\dots,a_n)$, by $(a_i)_{i=1}^n$, or for short just by $(a_i)$ when the index $i$ and the length of the sequence $n$ are clear from the context. As in the previous paragraph, given $a=(a_1,\dots,a_n)\in A^n$ we may denote $a[i]=a_i$ for $1\leq i\leq n$. We also denote $a[i,j]=(a_i,a_{i+1}\dots,a_j)$ for $i,j\in\{1,2,\dots,n\}$: this is the empty sequence when $i>j$. 

By substituting $\R$ for $A$ in the previous paragraph we get the set of $n$-dimensional real vectors. We define a partial order for $v,w\in\R^n$ by $v\leq w$ if $v[i]\leq w[i]$ for $1\leq i\leq n$. A stronger inequality is $v\ll w$, which means that $v[i]<w[i]$ for all $1\leq i\leq n$. For any $x\in\R$ we define the constant vector $\cvec{x}=(x,\dots,x)$, so in particular $\cvec{0}=(0,\dots,0)$ and $\cvec{1}=(1,\dots,1)$. The Kronecker delta is defined by $\delta_{ij}=0$ when $i\neq j$ and $\delta_{ij}=1$ when $i=j$. Using this, the standard basis vectors $e_k\in\R^n$ for $1\leq k\leq n$ are defined by $e_k[i]=\delta_{ki}$ for $1\leq i\leq n$. In matrix multiplications we interpret vectors as column vectors.

To present some of the main results of this paper we will use the terminology of topological dynamical systems, one reference for these is~\cite{Kur03}.

\begin{definition}If $X$ is a compact metrizable topological space and $T:X\to X$ is a continuous map, we say that $(X,T)$ is a \emph{(topological) dynamical system}.\end{definition}

As a particular example, for a finite set $\alp$ (an alphabet), $d\in\Zpos$ and $z\in\Z^d$ we define the shift maps $\sigma_{z}:\alp^{\Z^d}\to \alp^{\Z^d}$ by $\sigma_{z}(f)[z']=f[z'+z]$ for all $f\in \alp^{\Z^d}$ and $z'\in\Z^d$. Then $(\alp^{\Z^d},\sigma_z)$ is a dynamical system. Whenever $X\subseteq \alp^{\Z^d}$ is closed and $\sigma_z(X)=X$, then $(X,\sigma_z)$ is also a dynamical system. Note that we used the same notation for both $\sigma_z$ and its restriction to $X$: in practice this will not cause confusion. Let us also mention that if $\sigma_z(X)=X$ for all $z\in\Z^d$, then $X$ in fact becomes a system with multidimensional dynamics called a $d$-dimensional subshift. In the one-dimensional case we define $\sigma:\alp^\Z\to\alp^\Z$ by $\sigma=\sigma_1$ and for a closed subset $X\subseteq\alp^\Z$ satisfying $\sigma(X)=X$ we call $(X,\sigma)$ a (one-dimensional) subshift. One can also define $\sigma:\alp^\N\to\alp^\N$ by $\sigma(x)[i]=x[i+1]$ for $i\in\N$, and then $(X,\sigma)$ for a closed $X\subseteq\alp^\N$ satisfying $\sigma(X)\subseteq \alp^\N$ is called a one-sided subshift.

The structure preserving transformations between topological dynamical systems are known as morphisms.

\begin{definition}We write $\psi:(X,T)\to (Y,S)$ whenever $(X,T)$ and $(Y,S)$ are dynamical systems and $\psi:X\to Y$ is a continuous map such that $\psi\circ T=S\circ\psi$ (this equality is known as the \emph{equivariance condition}). Then we say that $\psi$ is a \emph{morphism}. If $\psi$ is surjective, we say that $\psi$ is a \emph{factor map} and that $(Y,S)$ is a \emph{factor} of $(X,T)$ (via $\psi$). If $\psi$ is bijective, we say that $\psi$ is a \emph{conjugacy} and that $(X,T)$ and $(Y,S)$ are \emph{conjugate} (via $\psi$).\end{definition}

The particular dynamical systems under consideration in this paper will be shift maps on tilings and cellular automata.

\section{Mixed base representations of numbers}\label{mixedBaseSect}

A vector $m\in\Zpos^k$ such that $m[i]$ divides $m[i+1]$ for $1\leq i<k$ is called a \emph{mixed base}. A number $a\in\N$ has a mixed base-$m$ representation $a=\sum_{i=0}^k a_i m[i]$, (with the convention $m[0]=1$) where $a_i$ are the unique integers satisfying $0\leq a_i<m[i+1]/m[i]$ for $0\leq i<k$ and $a_k\geq 0$. We define $\base(a,n)=(a_0,a_1,\dots, a_k)$, so in particular $\base(a,n)[i]=a_{i-1}$ for $1\leq i\leq k+1$. For example, the base-10 representation (in the usual sense) for $a\in\N$ arises by choosing $m=(10,10^2,\dots,10^k)$ for a sufficiently large $k$.

We may also use a single vector $n\in\Zpos^d$, a \emph{prebasis}, to form many different mixed bases as various products of $n[i]$. For $v\in\Z^d$ we let
\[m(n,v)=\prod_{j=1}^d n[j]^{-v[j]}\]
(this is a positive integer in the special case $v\leq\cvec{0}$). For a \emph{directive sequence}, a decreasing sequence of vectors $(v_i)_{i=1}^k$ satisfying $\cvec{0}\geq v_1\geq v_2\geq\cdots\geq v_k\in\Z^d$ (a \emph{binary} directive sequence in the case $v_i\in\{-1,0\}^d$), we define a mixed base $m(n,(v_{i}))\in\Zpos^k$ by
\[m(n,(v_{i}))[j]=m(n,v_j)\mbox{ for }1\leq j\leq k.\]
We also define $\base_{n}(a,(v_1,\cdots,v_k))=\base(a,m(n,(v_i)))$ for $a\in\N$.

\begin{remark}
Our definition of $m(n,v)$ where a negative coordinate $v[j]$ corresponds to raising $n[j]$ to a positive power may seem strange. Throughout this paper negative numbers will appear at points where positive numbers could be expected. Our conventions are ultimately motivated by the peculiarity that the usual method of writing down the digits of a number causes the more significant digits to appear to the left (i.e. towards the negative direction of the $x$-axis) instead of to the right.
\end{remark}

We also mention some telescoping properties for the base representations. For $x\in A^n$ and $y\in A^m$ with any set $A$ define the insertion $x\ins_i y\in A^{n+m}$ for $0\leq i\leq n$ by 
\begin{flalign*}
&(x\ins_i y)[1,i]=x[1,i] \\
&(x\ins_i y)[i+1,i+m]=y \\
&(x\ins_i y)[i+m+1,n+m]=x[i+1,n]
\end{flalign*}
and the overwrite $x\insover_i y\in A^{n+m-1}$ for $1\leq i\leq n$ by
\begin{flalign*}
&(x\insover_i y)[1,i-1]=x[1,i-1] \\
&(x\insover_i y)[i,i+m-1]=y \\
&(x\insover_i y)[i+m,n+m-1]=x[i+1,n].
\end{flalign*}

\begin{lemma}\label{scope}
Let $m\in\Zpos^k$, $m'\in\Zpos^{k'}$ and $m\ins_i m'\in\Zpos^{k+k'}$ be mixed bases with $0\leq i\leq k$. If $a\in\N$ and $\base(a,m)=(a_0,\dots,a_{k})$, then (with the convention that $m[0]=1$)
\[\base(a,m\ins_i m')=\base(a,m)\insover_{i+1}\base(a_i,m'/m[i]).\]
\end{lemma}
\begin{proof}
Let $\base(a_i,m'/m[i])=(b_0,\dots,b_{k'})$. By assumption $a=\sum_{j=0}^k a_j m[j]$ ($0\leq a_j<m[j+1]/m[j]$ for $0\leq j<k$, $a_k\geq 0$) and $a_i=b_0+\sum_{j=1}^{k'} b_j(m'[j]/m[i])$ ($0\leq b_0<m'[1]/m[i]$, $0\leq b_j<m'[j+1]/m'[j]$ for $1\leq j<k'$, $b_k\geq 0$). By substituting the base-$m'$ representation of $a_i$ into the base-$m$ representation of $a$ we find that
\[a=\sum_{j=0}^{i-1}a_j m[j]+b_0m[i]+\sum_{j=1}^{k'}b_j m'[j]+\sum_{j=i+1}^k a_j m[j].\]
We need to check that the coefficients of $m[j]$ and $m'[j]$ are such that 
\[\base(a,m\ins_i m')=(a_0,\dots a_{i-1}, b_0,\dots,b_{k'},a_{i+1},\dots,a_k):\]
everything except the inequality $b_{k'}<m[i+1]/m'[k']$ (when $i<k$) follows directly from the definition of $a_j$ and $b_j$. Because $a_i<m[i+1]/m[i]$, it follows that $b_{k'}\leq a_i m[i]/m'[k']<m[i+1]/m'[k']$, and we are done. 
\end{proof}

\begin{lemma}\label{prescope}
Let $n$ be a prebasis and let $(v_j)_{j=1}^{k}$, $(w_j)_{j=1}^{k'}$ and $(v_j)_{j=1}^{k}\ins_i (w_j)_{j=1}^{k'}$ be directive sequences with $0\leq i\leq k$. If $a\in\N$ and $\base_{n}(a,(v_j)_{j=1}^k)=(a_0,\dots,a_k)$, then (with the convention that $v_0=\cvec{0}$)
\[\base_{n}(a,(v_j)_{j=1}^k\ins_i(w_j)_{j=1}^{k'})=\base_{n}(a,(v_j)_{j=1}^k)\insover_{i+1}\base_{n}(a_i,(w_j-v_i)_{j=1}^{k'}).\]
\end{lemma}
\begin{proof}
Let $M=m(n,v_i)$. We note that for $1\leq j'\leq k'$,
\begin{flalign*}
&m(n,w_{j'}-v_i)=\prod_{j=1}^d n[j]^{(-w_{j'}+v_i)[j]} \\
&=\prod_{j=1}^d n[j]^{-w_{j'}[j]}\prod_{j=1}^d n[j]^{v_i[j]}=m(n,w_{j'})/M.
\end{flalign*}
By assumption $\base(a,m(n,(v_j)_{j=1}^{k}))=\base_{n}(a,(v_j)_{j=1}^k)=(a_0,\dots,a_k)$ and
\begin{flalign*}
&\base_{n}(a,(v_j)_{j=1}^k\ins_i(w_j)_{j=1}^{k'})=\base(a,m(n,(v_j)\ins_i(w_j))) \\
&=\base(a,m(n,(v_j)_{j=1}^{k})\ins_i m(n,(w_j)_{j=1}^{k'})) \\
\overset{L.~\ref{scope}}&{=}\base(a,m(n,(v_j)_{j=1}^{k}))\insover_{i+1}\base(a_i,m(n,(w_j)_{j=1}^{k'})/M)) \\
&=\base_{n}(a,(v_j)_{j=1}^k)\insover_{i+1}\base_{n}(a_i,(w_j-v_i))_{j=1}^{k'}).
\end{flalign*}
\end{proof}

\begin{lemma}\label{endscope}
Let $n$ be a prebasis, let $(v_j)_{j=1}^k$ and $(w_j)_{j=1}^{k'}$ be directive sequences and let $0\leq i<k$, $0\leq i'<k'$ be such that $v_i=w_{i'}$ and $v_{i+1}=w_{i'+1}$ (with the convention $v_0=w_0=\cvec{0}$). For any $a\in\N$ it holds that
\[\base_{n}(a,(v_j))[i+1]=\base_{n}(a,(w_j))[i'+1].\]
If $v_k=w_{k'}$, then
\[\base_{n}(a,(v_j))[k+1]=\base_{n}(a,(w_j))[k'+1].\]
\end{lemma}
\begin{proof}
For the first claim of the lemma concerning $0\leq i<k$, $0\leq i'<k'$ it suffices to show that $\base_{n}(a,(v_i,v_{i+1}))[2]=\base_{n}(a,(v_j))[i+1]$, because from this it follows that
\begin{flalign*}
&\base_{n}(a,(v_j))[i+1]=\base_{n}(a,(v_i,v_{i+1}))[2] \\
&=\base_{n}(a,(w_{i'},w_{i'+1}))[2]=\base_{n}(a,(w_j))[i'+1].
\end{flalign*}
Let $\base_{n}(a,(v_i,v_{i+1}))=(a_0,a_1,a_2)$. By one application of Lemma~\ref{prescope}
\begin{flalign*}
&\base_{n}(a,(v_1,\dots,v_i,v_{i+1}))=\base_{n}(a,(v_i,v_{i+1})\ins_0(v_1,\dots,v_{i-1})) \\
&=(a_0,a_1,a_2)\insover_1(b_0,\dots,b_{i-1})=(b_0,\dots,b_{i-1},a_1,a_2)
\end{flalign*}
for some $b_j\in\N$. By another application of the same lemma
\begin{flalign*}
&\base_{n}(a,(v_j))=\base_{n}(a,(v_1,\dots,v_{i+1})\ins_{i+1}(v_{i+2},\dots,v_k)) \\
&=(b_0,\dots,b_{i-1},a_1,a_2)\insover_{i+2}(c_{i+1},\dots,c_k)=(b_0,\dots,b_{i-1},a_1,c_{i+1},\dots,c_k)
\end{flalign*}
for some $c_j\in\N$. Therefore $\base_{n}(a,(v_i,v_{i+1}))[2]=a_1=\base_{n}(a,(v_j))[i+1]$.

In a similar manner, for the last claim of the lemma it suffices to show that $\base_{n}(a,(v_k))[2]=\base_{n}(a,(v_j))[k+1]$. Let $\base_n(a,(v_k))=(a_0,a_1)$. By Lemma~\ref{prescope}
\begin{flalign*}
&\base_{n}(a,(v_1,\dots,v_{k-1},v_{k}))=\base_{n}(a,(v_{k})\ins_0(v_1,\dots,v_{k-1})) \\
&=(a_0,a_1)\insover_1(b_0,\dots,b_{k-1})=(b_0,\dots,b_{k-1},a_1)
\end{flalign*}
for some $b_j\in\N$. Therefore $\base_{n}(a,(v_k))[2]=a_1=\base_{n}(a,(v_j))[k+1]$.
\end{proof}

\begin{lemma}\label{twostepbase}
let $n$ be a prebasis, let $a\in\N$ and let $p_1,p_2,q_1,q_2\in\Z^d$ be such that $q_1\geq q_2$ and $p_1\geq p_1+q_1\geq p_1+q_2\geq p_2$ are directive sequences. If $\base_{n}(a,(p_1,p_2))=(a_0,a_1,a_2)$, then 
\[\base_{n}(a,(p_1+q_1,p_1+q_2))[2]=\base_{n}(a_1,(q_1,q_2))[2].\]
\end{lemma}
\begin{proof}
By applying Lemma~\ref{prescope}
\begin{flalign*}
&\base_{n}(a,(p_1,p_1+q_1,p_1+q_2,p_2)) \\
&=\base_{n}(a,(p_1,p_2))\insover_2\base_{n}(a_1,(q_1,q_2)),
\end{flalign*}
and therefore by Lemma~\ref{endscope}
\begin{flalign*}
&\base_{n}(a,(p_1+q_1,p_1+q_2))[2] \\
&=\base_{n}(a,(p_1,p_1+q_1,p_1+q_2,p_2))[3]=\base_{n}(a_1,(q_1,q_2))[2].
\end{flalign*}
\end{proof}

For a map $\emb:\{1,\dots,d'\}\to\{1,\dots,d\}$ and $p\in\Z^d$ there is an affine map $\Emb_{p,\emb}:\R^{d'}\to \R^d$ defined by $\Emb_{p,\emb}(\cvec{0})=p$ and $\Emb_{p,\emb}(e_i)=p+e_{\emb(i)}$ for $1\leq i\leq d'$.

\begin{lemma}\label{subbasis}
Let $\emb:\{1,\dots,d'\}\to\{1,\dots,d\}$ be injective and let $n$, $n'$ be prebases of dimensions $d$ and $d'$ satisfying $n'[i]=n[\emb(i)]$ for $1\leq i\leq d'$. For a $d'$-dimensional directive sequence $v_1\geq\cdots\geq v_k$, for $w_i=\Emb_{\cvec{0},\emb}(v_i)$ and $a\in\N$ it holds that $\base_{n'}(a,(v_1,\dots,v_k))=\base_{n}(a,(w_1,\dots,w_k))$.
\end{lemma}
\begin{proof}
Observe that for any $v\in\Z^{d'}$, $1\leq j'\leq d'$ and $j=\emb(j')$ it holds that $\Emb_{\cvec{0},\emb}(v)[j]=v[j']$. If $j$ is not in the image of $\emb$, then $\Emb_{\cvec{0},\emb}(v)[j]=0$. Using the substitution $j=\emb(j')$ for $0\leq j'\leq d'$ below, we can compute for $1\leq i\leq k$ that
\begin{flalign*}
&m(n,w_i)=m(n,\Emb_{\cvec{0},\emb}(v_i))=\prod_{j=1}^d n[j]^{-\Emb_{\cvec{0},\emb}(v_i))[j]} \\
&=\prod_{j'=1}^{d'}n[\emb(j')]^{-\Emb_{\cvec{0},\emb}(v_i)[\emb(j')]}=\prod_{j'=1}^{d'}n'[j']^{-v_i[j']}=m(n',v_i),
\end{flalign*}
and therefore
\[\base_{n}(a,(w_i))=\base(a,m(n,(w_i))=\base(a,m(n'(v_i)))=\base_{n'}(a,(v_i)).\]
\end{proof}

\section{Multiplication cubes}\label{WangSect}

\subsection{Preliminaries on multiplication cubes}\label{cubePre}

Our model for the $d$-dimensional hypercube is the set $[-1,0]^d$. The bottom and top $(d-1)$-dimensional hyperfaces \emph{orthogonal} to $e_i$ are
\[\boti_i=\{v\in [-1,0]^d\mid v[i]=-1\}\mbox{ and }\topi_i=\{v\in [-1,0]^d\mid v[i]=0\}\]
respectively.

One can take a collection of hypercubes, color their $(d-1)$-dimensional hyperfaces $\boti_i$ and $\topi_i$ with various colors, and form tessellations of the space with these cubes in such a way that the adjacent hyperfaces of neighboring hypercubes have matching colors. An example of such a colored set of 2-dimensional cubes and a part of a 2-dimensional tessellation is presented in Figure~\ref{tiling}. More precisely (and more abstractly) this is done as in the following definition.

\begin{figure}[ht]
\centering
\begin{tikzpicture}[scale=1]

\tile[0,2,5](0,0); \tile[1,2,5](2,0);
\tile[2,2,5](0,-1.5); \tile[3,2,5](2,-1.5);
\tile[4,2,5](0,-3); \tile[5,2,5](2,-3);
\tile[6,2,5](0,-4.5); \tile[7,2,5](2,-4.5);
\tile[8,2,5](0,-6); \tile[9,2,5](2,-6);

\tile[0,2,5](5,-0.5);\tile[0,2,5](6,-0.5);\tile[0,2,5](7,-0.5);\tile[0,2,5](8,-0.5);\tile[0,2,5](9,-0.5);\tile[0,2,5](10,-0.5);
\tile[4,2,5](5,-1.5);\tile[8,2,5](6,-1.5);\tile[6,2,5](7,-1.5);\tile[2,2,5](8,-1.5);\tile[4,2,5](9,-1.5);\tile[8,2,5](10,-1.5);
\tile[0,2,5](5,-2.5);\tile[1,2,5](6,-2.5);\tile[3,2,5](7,-2.5);\tile[6,2,5](8,-2.5);\tile[2,2,5](9,-2.5);\tile[5,2,5](10,-2.5);
\tile[0,2,5](5,-3.5);\tile[0,2,5](6,-3.5);\tile[0,2,5](7,-3.5);\tile[1,2,5](8,-3.5);\tile[2,2,5](9,-3.5);\tile[5,2,5](10,-3.5);
\tile[0,2,5](5,-4.5);\tile[0,2,5](6,-4.5);\tile[0,2,5](7,-4.5);\tile[0,2,5](8,-4.5);\tile[0,2,5](9,-4.5);\tile[1,2,5](10,-4.5);
\tile[0,2,5](5,-5.5);\tile[0,2,5](6,-5.5);\tile[0,2,5](7,-5.5);\tile[0,2,5](8,-5.5);\tile[0,2,5](9,-5.5);\tile[0,2,5](10,-5.5);

\fill   (9,-1.5) circle (1mm);

\end{tikzpicture}
\caption{The multiplication cube set $T_{(2,5)}$ and a part of a tiling from $X_{(2,5)}=X_{T_{(2,5)}}$ (with the upper right corner of the cube at the origin marked by a black dot). Consecutive powers of two (starting with $4,8,16,\dots$) can be found in the tiling along diagonals that go from bottom left to top right: for an explanation of this, see Example~\ref{powersof2} or Proposition~\ref{infShiftMul}.}
\label{tiling}
\end{figure}
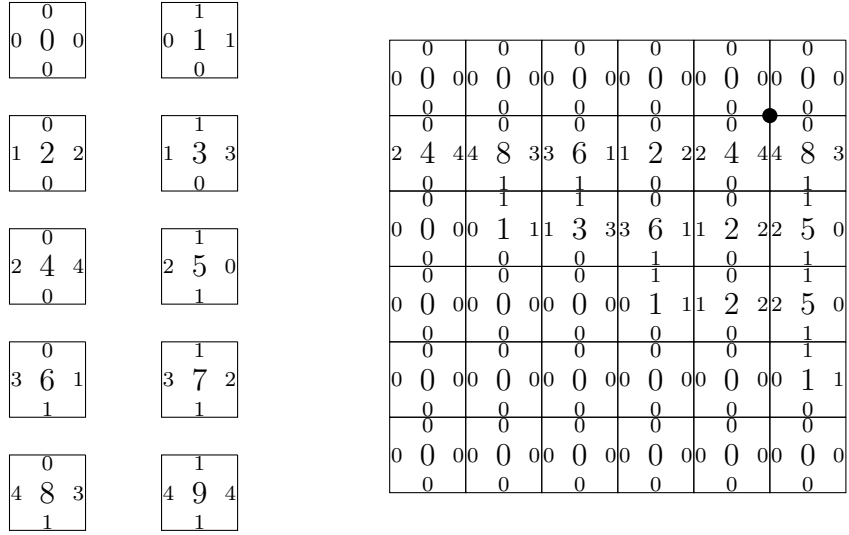

\begin{definition}
Let $T$ and $C$ be finite sets (the set of hypercubes and the set of colors) and for some $d\in\Zpos$ let $\boti_i:T\to C$ and $\topi_i: T\to C$ be functions for $1\leq i\leq d$ (the labeling functions). The set $T$  (with associated, implicit $C$, $\boti_i$ and $\topi_i$) is then called a Wang hypercube set. A valid tiling or tessellation over $T$ is a map $f:\Z^d\to T$ satisfying $\topi_i(f[z])=\boti_i(f[z+e_i])$ for all $z\in\Z^d$ and $1\leq i\leq d$. The set of all valid tilings over $T$ is denoted by $X_T$.
\end{definition}

The pair $(X_T,\sigma_z)$ is a dynamical system for all $z\in\Z^d$, and $X_T$ is in fact an example of a so-called subshift of finite type. Wang hypercubes and their tessellations are the most commonly studied in the case $d=2$, and particularly then the elements of $T$ are called Wang tiles. These originate from Wang's paper~\cite{Wang61}.

We will proceed to define a particular class of Wang hypercubes, the so-called multiplication cubes. These cubes will in fact have labels on all of their lower-dimensional hyperfaces, not just the $(d-1)$-dimensional ones. The set of $k$-dimensional hyperfaces within the $d$-dimensional hypercube is denoted by $S_{d,k}$ and the set of all hyperfaces is denoted by $S_d$. Given $s\in S_d$, the set of all hyperfaces that are subsets of $s$ is denoted by $S_d(s)$.

We develop two different notations to refer to individual hyperfaces. For the first notation, whenever $p,u\in\{-1,0\}^d$ are orthogonal (in other words, $p[i]=u[i]=-1$ cannot happen for $1\leq i\leq d$), we say that the hyperface at $p$ along $u$ is
\[\faceat(p,u)=\{v\in [-1,0]^d\mid v[i]=p[i]\mbox{ when }u[i]=0\}.\]

For the second notation, let
\[V_d=\{(v_1,v_2)\in (\{-1,0\}^d)^2\mid v_1\geq v_2\}.\]
For any $(v_1,v_2)\in V_d$, the minimal $k$-dimensional hyperface of $[-1,0]^d$ containing both $v_1$ and $v_2$ is $\faceat(v_1,v_2-v_1)$, where $k$ is the number of ones in $v_1-v_2$. This gives a one-to-one correspondence between $V_d$ and $S_d$.  

Let $n$ be a prebasis of dimension $d$. To any $a\in\digs_N$ with $N=\prod_{i=1}^d n[i]$ we will associate a $d$-dimensional Wang hypercube, a so-called multiplication cube $\cube_{n}(a)$, whose underlying structure is that of a function:
\[\cube_{n}(a):V_d\to\digs_N,\quad \cube_{n}(a)[v_1,v_2]=\base_{n}(a,(v_1,v_2))[2]\mbox{ for }(v_1,v_2)\in V_d.\]
The interpretation is that $\cube_{n}(a)[v_1,v_2]$ is the label of the hyperface spanned by the vectors $v_1$ and $v_2$: this way we get labels on all lower-dimensional hyperfaces. If $s\in S_d$ is the hyperface spanned by $v_1$ and $v_2$, we may also write $\cube_{n}(a)[v_1,v_2]=s(\cube_{n}(a))=s\cube_n(a)$. In particular, the labels of the bottom and top hyperfaces of $\cube_{n}(a)$ orthogonal to $e_i$ are
\[\boti_i(\cube_{n}(a))=\cube_{n}(a)[-e_i,-\cvec{1}], \quad \topi_i(\cube_{n}(a))=\cube_{n}(a)[\cvec{0},-\cvec{1}+e_i]\]
respectively. To write this out more explicitly, let $a=a_1n[i]+a_0$ be such that $\base(a,(N,n[i]))=(0,a_1,a_0)$ and let $a=a_2'\prod_{j\neq i}n[j]+a_1'$ be such that $\base(a,(\prod_{j\neq i}n[i],1))=(a_2',a_1',0)$. Then we see that
\[\boti_i(\cube_{n}(a))=a_1,\qquad \topi_i(\cube_{n}(a))=a_1'.\]
An example of a $3$-dimensional multiplication cube together with (some of) the labels of faces and edges is presented in Figure~\ref{cube}.

\begin{figure}[ht]
\centering
\begin{tikzpicture}[scale=0.6]

\cubefig[10,2,3,5](0,0);
\tilescale[10,3,5,3](4,-1);\tilescale[3,2,5,3](8,-1);\tilescale[4,2,3,3](12,-1);

\end{tikzpicture}
\caption{The multiplication cube $\cube_{(2,3,5)}(10)$ with the faces $\topi_1$ (right), $\boti_2$ (front) and $\topi_3$ (top) visible. One can also verify that these faces and their adjacent edges yield the lower-dimensional multiplication cubes $\cube_{(3,5)}(10)$, $\cube_{(2,5)}(3)$ and $\cube_{(2,3)}(4)$.}
\label{cube}
\end{figure}
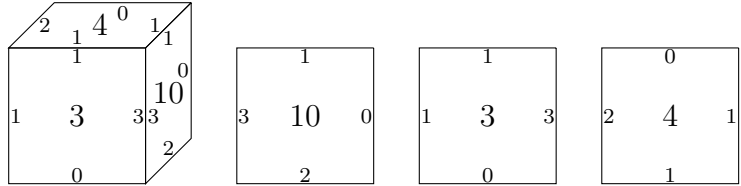

For a prebasis $n$ of dimension $d$ and $N=\prod_{i=1}^d n[i]$, the set of base-$n$ multiplication cubes and the valuations of cubes are defined as
\[T_{n}=\{\cube_{n}(a)\mid a\in\digs_N\}, \quad \val_{n}(t)=t[\cvec{0},-\cvec{1}]\mbox{ for }t\in T_n.\]
In other words, $\val_{n}(t)$ is the label of the unique $d$-dimensional hyperface of $t$. More importantly, this also equals the original value that was used to form the hypercube and the maps $\val_{n}$, $\cube_{n}$ are inverses of each other, because $a=0\cdot N+a\cdot 1+0\cdot 1$ and $\base(a,(N,1))=(0,a,0)$ for any $a\in\digs_N$, and therefore
\[\val_{n}(\cube_{n}(a))=\cube_{n}(a)[\cvec{0},-\cvec{1}]=a\mbox{ and }\cube_{n}(\val_{n}(\cube_{n}(a)))=\cube_{n}(a).\]
The set of valid tilings over $T_n$ is denoted by $X_n=X_{T_n}$.

Later we will see that calling the elements of $T_n$ multiplication cubes is justified. If a tiling contains a base-$N$ representation of a real number on its main diagonal, then shifting the tiling by $\sigma_{e_i}$ corresponds to multiplying this real number by $n[i]$ as will be seen in Proposition~\ref{infShiftMul}. The tilings will also be connected to so-called multiplication automata in Subsection~\ref{WangCASubSect}.

We observe that in the case of a one-dimensional prebasis $n=(N)$ the labels of top and bottom hyperfaces satisfy $\boti_1(\cube_{(N)}(a))=\topi_1(\cube_{(N)}(a))=0$ for all $a\in\digs_N$ and therefore $X_{(N)}=T_{(N)}^\Z$. For a higher dimensional example, Figure~\ref{tiling} shows the tile set $T_{(2,5)}$ and a part of a tiling from $X_{(2,5)}$. In the higher dimensional cases it is not necessarily obvious that there exist any valid tilings except the trivial one consisting only of copies of $\cube_n(0)$, but in Corollary~\ref{tessExist} we will see that any arrangement of multiplication cubes on the main diagonal of the $d$-dimensional space can be extended to a valid tessellation of the whole space.

\begin{lemma}\label{prebasispath}
For a directive sequence $(v_j)_{j=1}^k$ and $a\in\N$ it holds that
\[\base_{n}(a,(v_1,v_k))[2]=\sum_{i=1}^{k-1}m(n,v_{i}-v_1)\base_{n}(a,(v_i,v_{i+1}))[2].\]
\end{lemma}
\begin{proof}
The cases $k\leq 2$ are trivial. If $k\geq 3$, write $\base_{n}(a,(v_1,v_k))[2]=(a_0',a_1',a_k')$ and $\base_{n}(a,(v_j)_{j=1}^k)=(a_j)_{j=0}^k$. It follows that
\[a_0'+a_1'm(n,v_1)+a_k'm(n,v_k)=a=\sum_{i=0}^k a_i m(n,v_i)\]
(with the convention that $v_0=\cvec{0}$). From Lemma~\ref{endscope} we see that
\[a_i=\base_{n}(a,(v_j)_{j=1}^k)[i+1]=\base_{n}(a,(v_i,v_{i+1}))[2]\mbox{ for }0\leq i< k.\]
The first part of Lemma~\ref{endscope} shows that $a_0=a_0'$ and the last part of Lemma~\ref{endscope} shows that $a_k=a_k'$, so
\[a_1'=\sum_{i=1}^{k-1}a_i m(n,v_i)/m(n,v_1)=\sum_{i=1}^{k-1}m(n,v_i-v_1)\base_{n}(a,(v_i,v_{i+1}))[2].\]
\end{proof}

\begin{lemma}\label{cubePath}
For a binary directive sequence $(v_i)_{i=1}^k$ and $t\in T_{n}$ it holds that
\[m(n,v_1)t[v_1,v_k]=\sum_{j=1}^{k-1}m(n,v_j)t[v_j,v_{j+1}].\]
\end{lemma}
\begin{proof}
Let $t=\cube_{n}(a)$. We rewrite the statement of the lemma
\begin{flalign*}
&m(n,v_1)\cube_{n}(a)[v_1,v_k]=\sum_{j=1}^{k-1}m(n,v_j)\cube_{n}(a)[v_j,v_{j+1}] \\
\iff&m(n,v_1)\base_{n}(a,(v_1,v_k))[2]=\sum_{j=1}^{k-1}m(n,v_j)\base_{n}(a,(v_j,v_{j+1}))[2] \\
\iff&\base_{n}(a,(v_1,v_k))[2]=\sum_{j=1}^{k-1}m(n,v_j-v_1)\base_{n}(a,(v_j,v_{j+1}))[2],
\end{flalign*}
and the last equality holds by Lemma~\ref{prebasispath}.
\end{proof}

\subsection{Lower dimensional matchings in tilings by multiplication cubes}\label{lowerMatchSubSect}

Given a prebasis $n$ of dimension $d$, the $d-1$ dimensional hyperfaces match in a tessellation of $X_{n}$ by definition, but there are also lower dimensional matchings as we will see in Theorem~\ref{subfaceNeigh}.

Recall that any $s\in S_d$ can be represented in the form $s=\faceat(p,u)$ for some orthogonal $p,u\in\{-1,0\}^d$. If $u$ contains $d'$ non-zeroes and $\emb:\{1,\dots,d'\}\to\{1,\dots,d\}$ is an injection with the image set $\{i\mid u[i]=-1\}$, then the injective affine map $\Emb_{p,\emb}:\R^{d'}\to \R^d$ defined by $\Emb_{p,\emb}(\cvec{0})=p$ and $\Emb_{p,\emb}(e_i)=p+e_{\emb(i)}$ for $1\leq i\leq d'$ satisfies $\Emb_{p,\emb}([-1,0]^{d'})=s$.

\begin{lemma}\label{subface}
Let $s=\faceat(p,u)\in S_{d,d'}$, $s'\in S_{d'}$ and $t\in T_{n}$. Let $\emb:\{1,\dots,d'\}\to\{1,\dots,d\}$ be an injection with the image set $D=\{i\mid u[i]=-1\}$ and let $n'\in\Zpos^{d'}$ be defined by $n'[i]=n[\emb(i)]$ for $1\leq i\leq d'$. Then $(\Emb_{p,\emb}(s'))(t)=s'(\cube_{n'}(s(t)))$.
\end{lemma}
\begin{proof}
Let $s'=\faceat(p',u')$ and $t=\cube_{n}(a)$ (it follows that $\Emb_{p,\emb}(p')\leq \Emb_{p,\emb}(\cvec{0})$ and $\Emb_{p,\emb}(p'+u')\geq \Emb_{p,\emb}(-\cvec{1})$). Let $\base_{n}(a,(\Emb_{p,\emb}(\cvec{0}),\Emb_{p,\emb}(-\cvec{1})))=(a_0,a_1,a_2)$ and $\base_{n}(a_1,(\Emb_{p,\emb}(p')-\Emb_{p,\emb}(\cvec{0}),\Emb_{p,\emb}(p'+u')-\Emb_{p,\emb}(\cvec{0}))=(b_0,b_1,b_2)$. Then by Lemma~\ref{prescope}
\[\base_{n}(a,(\Emb_{p,\emb}(\cvec{0}),\Emb_{p,\emb}(p'),\Emb_{p,\emb}(p'+u'),\Emb_{p,\emb}(-\cvec{1})))=(a_0,b_0,b_1,b_2,a_2).\]

From this we can see that
\begin{flalign*}
(\Emb_{p,\emb}(s'))(t)&=t[\Emb_{p,\emb}(p'),\Emb_{p,\emb}(p'+u')]=\base_{n}(a,(\Emb_{p,\emb}(p'),\Emb_{p,\emb}(p'+u')))[2] \\
\overset{L.~\ref{endscope}}&{=}\base_{n}(a,(\Emb_{p,\emb}(\cvec{0}),\Emb_{p,\emb}(p'),\Emb_{p,\emb}(p'+u'),\Emb_{p,\emb}(\cvec{-1})))[3]=b_1.
\end{flalign*}

On the other hand,
\[s(t)=\base_{n}(a,(\Emb_{p,\emb}(\cvec{0}),\Emb_{p,\emb}(-\cvec{1})))[2]=a_1\]
and
\begin{flalign*}
s'(\cube_{n'}(s(t)))&=\cube_{n'}(s(t))[p',p'+u']=\base_{n'}(s(t),(p',p'+u'))[2] \\
&=\base_{n'}(a_1,(p',p'+u'))[2] \\
\overset{L.~\ref{subbasis}}&{=}\base_{n}(a_1,(\Emb_{\cvec{0},\emb}(p'),\Emb_{\cvec{0},\emb}(p'+u')))[2]=b_1.
\end{flalign*}
We conclude that
\[(\Emb_{p,\emb}(s'))(t)=b_1=s'(\cube_{n'}(s(t))).\]
\end{proof}

\begin{lemma}\label{subfaceSame}
Let $s_1,s_2\in S_d$, let $u\in\Z^d$ be such that $s_2=s_1+u$ and let $t_1,t_2\in T_{n}$ be such that $s_1(t_1)=s_2(t_2)$. Then for any $s_1'\in S_d(s_1)$ and $s_2'=s_1'+u\in S_d(s_2)$ it holds that $s_1'(t_1)=s_2'(t_2)$.
\end{lemma}
\begin{proof}
Let $s_1=\faceat(p,v)$, which means that $s_2=\faceat(p+u,v)$. Assume that the cardinality of $D=\{i\mid v[i]=1\}$ is $d'$ and let $\iota:\{1,\dots,d'\}\to\{1,\dots,d\}$ be an injection with image $D$. Let $n'\in\N^{d'}$ be defined by $n'[i]=n[\emb(i)]$ for $1\leq i\leq d'$. Let $s'\in S_{d'}$ be such that $\Emb_{p,\emb}(s')=s_1'$, which means that $\Emb_{p+u,\emb}(s')=s_2'$. By Lemma~\ref{subface}
\begin{flalign*}
s_1'(t_1)&=\Emb_{p,\emb}(s')(t_1)=s'(\cube_{n'}(s_1(t_1))) \\
&=s'(\cube_{n'}(s_2(t_2)))=\Emb_{p+u,\emb}(s')(t_2)=s_2'(t_2).
\end{flalign*}
\end{proof}

\begin{theorem}\label{subfaceNeigh}
Let $f\in X_{n}$ and let $s,s+v\in S_d$ for some $v\in\Z^d$. For any $z\in \Z^d$ it holds that $(s+v)(f[z])=s(f[z+v])$.
\end{theorem}
\begin{proof}
Since $v$ is a sum of vectors in the standard basis, the claim follows by induction if we show that $(s+e_i)(f[z])=s(f[z+e_i])$ for $1\leq i\leq d$ such that $s,s+e_i\in S_d$. We know that $s+e_i\in S_d(\topi_i)$ and $s\in S_d(\boti_i)$. From $f\in X_{n}$ it follows that $\topi_i(f[z])=\boti_i(f[z+e_i])$ and Lemma~\ref{subfaceSame} implies that $(s+e_i)(f[z])=s(f[z+e_i])$.
\end{proof}

The results of this section allow us to speak of values on edges in a tiling without referring to individual cubes. To be more precise, we imagine $\Z^d$ to be the vertex set of an infinite directed graph, the $d$-dimensional grid, having an edge $(z-e_i,z)$ (the edge from $z-e_i$ to $z$) for every $z\in\Z^d$ and $1\leq i\leq d$. When we consider a tiling $f\in X_{n}$ and say that the cube $f[z]$ is at position $z\in\Z^d$, we visualize this as the unit cube $[-1,0]^d$ whose ``top-right'' corner, the point $\cvec{0}$, is positioned at $z$. Then the edges of the cube $f[z]$ overlap with some of the edges of the $d$-dimensional grid: an edge of the form $s=\faceat(p,-e_i)$ overlaps the directed edge $\dedge_z(s)=(z+p-e_i,z+p)$ in the grid. We also define the undirected edge $\edge_z(s)=\{z+p-e_i,z+p\}$ and the (undirected) label of the undirected edge $\edge_z(s)$ in the tessellation by
\[\lbl(f,\edge_z(s))=sf[z].\]
To see that this is well defined, let $z_1,z_2\in\Z^d$ and $s_1=\faceat(p_1,-e_i)$, $s_2=\faceat(p_2,-e_i)$ be such that $\edge_{z_1}(s_1)=\edge_{z_2}(s_2)$. Then in particular $z_1+p_1=z_2+p_2$ and $s_1=s_2+(p_1-p_2)$. Because of this,
\begin{flalign*}
\lbl(f,\edge_{z_1}(s_1))&=s_1 f[z_1]=(s_2+(p_1-p_2))f[z_2-(p_1-p_2)] \\
\overset{T.~\ref{subfaceNeigh}}&{=}s_2 f[z_2]=\lbl(f,\edge_{z_2}(s_2)).
\end{flalign*}

We note that each edge $\{z-e_i,z\}$ has a label, because $\edge_z(\faceat(\cvec{0},-e_i))=\{z-e_i,z\}$. For an example of how a tessellation yields labels on the directed edges, see the left and middle parts of Figure~\ref{tiletogrid} (for the drawn part of the grid, the definition can be applied in the form $\lbl(f,\edge_{\cvec{0}}(s))=sf[\cvec{0}]$).

\subsection{Path integrals over tilings}\label{pathSubSect}

Given a prebasis $n$ of dimension $d$, we assign to every point $v\in\Z^d$ a weight by $\wgt_{n}(v)=m(n,v)$. A path (of length $k-1$) is a sequence $P=(P_i)_{i=1}^k\in(\Z^d)^k$, where for $1\leq i< k$ we have $P_{i+1}-P_i=ae_j$ for some $a\in\{-1,1\}$, $1\leq j\leq d$. For $1\leq	i\leq k-1$, let $E_P(i)=P_{i+1}-P_i$. Given a valid tiling $f\in X_{n}$ and a pair of vectors $p,p'\in\Z^d$ satisfying $p'-p=ae_j$ for some $a\in\{-1,1\}$, $1\leq j\leq d$ we denote $\sign(p'-p)=a$ and
\[(p,p')f=\sign(p'-p)\wgt_{n}(\max\{p,p'\})\lbl(f,\{p,p'\}).\]
Then the path integral of $f$ over $P$ is
\[Pf=\sum_{i=1}^{k-1} (P_{i},P_{i+1})f.\]
See the right part of Figure~\ref{tiletogrid} containing labels given by a tessellation together with the weights of the points. Computing the path integral over a path of length $1$ consists of looking at the corresponding edge, multiplying its (undirected) label with the weight at its end point (now considered as a directed edge), and possibly multiplying the result by $-1$ in the case when the direction of the path opposes the direction of the edge.

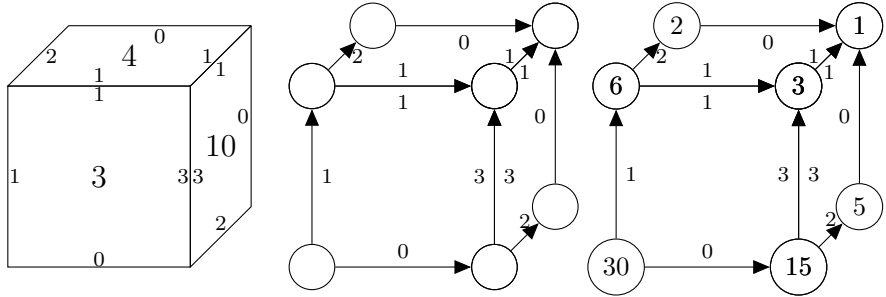
\begin{figure}[ht]
\centering
\begin{tikzpicture}[scale=0.8]

\cubefig[10,2,3,5](0,0);
\cubegrid[10,2,3,5](5,0);
\cubegridwt[10,2,3,5,1](10,0);

\end{tikzpicture}
\caption{Left: A tessellation with $\cube_{(2,3,5)}(10)$ positioned at the origin in $\Z^3$. \\
Middle: $\Z^3$ as a directed graph together with labels given by $\cube_{(2,3,5)}(10)$. \\
Right: The weights $\wgt_{(2,3,5)}(v)$ of the points $v\in\Z^3$ have been added to the grid, the point $\cvec{0}$ is at the upper right with $\wgt_{(2,3,5)}(\cvec{0})=1$.}
\label{tiletogrid}
\end{figure}

For $v\in\Z^d$ we denote $P+v=(P_i+v)_{i=1}^k$. We prove an equality connecting the integral of a tessellation $f$ over a shifted path $P+v$ to the integral of a shifted tessellation $\sigma_v(f)$ over a path $P$.

\begin{lemma}\label{pathshift}
For a path $P$, $f\in X_{n}$ and $v\in\Z^d$ it holds that $(P+v)f=\wgt_{n}(v)P\sigma_v(f)$.
\end{lemma}
\begin{proof}
We denote $a_i=\sign((P_{i+1}+v)-(P_i+v))=\sign(P_{i+1}-P_i)$ and directly compute that
\begin{flalign*}
(P+v)f&=\sum_{i=1}^{k-1}a_i\wgt_{n}(\max\{P_i+v,P_{i+1}+v\})\lbl(f,\{P_i+v,P_{i+1}+v\})\\
&=\sum_{i=1}^{k-1}a_i\wgt_{n}(v)\wgt_{n}(\max\{P_i,P_{i+1}\})\lbl(\sigma_v(f),\{P_i,P_{i+1}\}) \\
&=\wgt(v)P\sigma_v(f).
\end{flalign*}
\end{proof}

Next we note a simple cancellation property.

\begin{lemma}\label{cancel}
Let $f\in X_{n}$ and let $p,p'\in\Z^d$ satisfy $p'-p=ae_i$ for some $a\in\{-1,1\}$, $1\leq i\leq d$. Then $(p,p')f+(p',p)f=0$.
\end{lemma}
\begin{proof}
We may compute
\begin{flalign*}
&(p,p')f+(p',p)f \\
&=a\wgt_{n}(\max\{p,p'\})\lbl(f,\{p,p'\})-a\wgt_{n}(\max\{p',p\})\lbl(f,\{p',p\})=0.
\end{flalign*}
\end{proof}

In the following lemma we will show that the path integral over two consecutive edges around a square is equal to the path integral over the other two consecutive edges of the same square, see Figure~\ref{squarepath}. This may also be tested on the concrete examples of labels and weights in Figure~\ref{tiletogrid}.

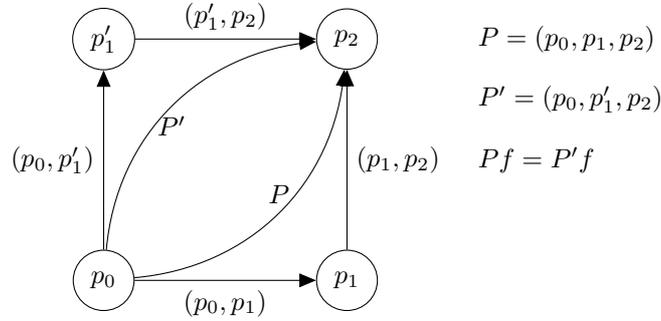
\begin{figure}[ht]
\centering
\begin{tikzpicture}[scale=0.8]

\node(sw) at (-4,-4) [shape=circle,draw,minimum size=8mm] {$p_0$};
\node(se) at (0,-4) [shape=circle,draw,minimum size=8mm] {$p_1$};
\node(ne) at (0,0) [shape=circle,draw,minimum size=8mm] {$p_2$};
\node(nw) at (-4,0) [shape=circle,draw,minimum size=8mm] {$p_1'$};

\draw[-{triangle 45}] (sw) to [below] node {$(p_0,p_1)$} (se);
\draw[-{triangle 45}] (se) to [right] node {$(p_1,p_2)$} (ne);
\draw[-{triangle 45}] (sw) to [left] node {$(p_0,p_1')$} (nw);
\draw[-{triangle 45}] (nw) to [above] node {$(p_1',p_2)$} (ne);

\draw[-{triangle 45}] (sw) to [bend left = 40, below] node {$P'$} (ne);
\draw[-{triangle 45}] (sw) to [bend right = 40, above] node {$P$} (ne);

\node at (2,0) [anchor=west] {$P=(p_0,p_1,p_2)$};
\node at (2,-1) [anchor=west] {$P'=(p_0,p_1',p_2)$};
\node at (2,-2) [anchor=west] {$Pf=P'f$};

\end{tikzpicture}
\caption{Let $P$, $P'$ be the two paths between two opposite corners of a square. The path integrals of a tessellation $f$ over $P$ and $P'$ are equal.}
\label{squarepath}
\end{figure}

\begin{lemma}\label{switch}
Let $f\in X_{n}$ and let $p_0,p_1,p_2\in\Z^d$ satisfy $p_1-p_0=a_1e_{i_1}$ and $p_2-p_1=a_2e_{i_2}$ for some $a_1,a_2\in\{-1,1\}$ and $1\leq i_1\neq i_2\leq d$. If $p_1'=p_0+(p_2-p_1)$, then $(p_0,p_1')f$ and $(p_1',p_2)f$ are defined and $(p_0,p_1)f+(p_1,p_2)f=(p_0,p_1')f+(p_1',p_2)f$.
\end{lemma}
\begin{proof}
We first note that $p_1'-p_0=p_2-p_1=a_2e_{i_2}$ and $p_2-p_1'=p_1-p_0=a_1e_{i_1}$. Therefore $(p_0,p_1')f$ and $(p_1'+p_2)f$ are defined. To prove the lemma it is sufficient to consider the case $a_1=a_2=1$ (the case in Figure~\ref{squarepath}). To see this, we show how to reduce the other cases to this one.

Case 1: $a_1=1, a_2=-1$. Let $q_0=p_1'$, $q_1=p_2$, $q_2=p_1$ and $q_1'=q_0+(q_2-q_1)=p_1'+(p_1-p_2)=p_0$. Then $q_1-q_0=p_2-p_1'=a_1e_{i_1}$, $q_2-q_1=p_1-p_2=(-a_2)e_{i_2}$ and
\begin{flalign*}
&(p_0,p_1)f+(p_1,p_2)f=(p_0,p_1')f+(p_1',p_2)f \\
\iff& (q_1',q_2)f+(q_2,q_1)f=(q_1',q_0)f+(q_0,q_1)f \\
\overset{L.~\ref{cancel}}{\iff}&(q_0,q_1')f+(q_1',q_2)f=(q_0,q_1)f+(q_1,q_2)f.
\end{flalign*}

Case 2: $a_1=-1,a_2=1$. Let $q_0=p_1$, $q_1=p_0$, $q_2=p_1'$ and $q_1'=q_0+(q_2-q_1)=p_1+(p_1'-p_0)=p_2$. Then $q_1-q_0=p_0-p_1=(-a_1)e_{i_1}$, $q_2-q_1=p_1'-p_0=a_2e_{i_2}$ and
\begin{flalign*}
&(p_0,p_1)f+(p_1,p_2)f=(p_0,p_1')f+(p_1',p_2)f \\
\iff& (q_1,q_0)f+(q_0,q_1')f=(q_1,q_2)f+(q_2,q_1')f \\
\overset{L.~\ref{cancel}}{\iff}&(q_0,q_1')f+(q_1',q_2)f=(q_0,q_1)f+(q_1,q_2)f.
\end{flalign*}

Case 3: $a_1=a_2=-1$. Let $q_0=p_2$, $q_1=p_1'$, $q_2=p_0$ and $q_1'=q_0+(q_2-q_1)=p_2+(p_0-p_1')=p_1$. Then $q_1-q_0=p_1'-p_2=(-a_1)e_{i_1}$, $q_2-q_1=p_0-p_1'=(-a_2)e_{i_2}$ and
\begin{flalign*}
&(p_0,p_1)f+(p_1,p_2)f=(p_0,p_1')f+(p_1',p_2)f \\
\iff& (q_2,q_1')f+(q_1',q_0)f=(q_2,q_1)f+(q_1,q_0)f \\
\overset{L.~\ref{cancel}}{\iff}&(q_0,q_1)f+(q_1,q_2)f=(q_0,q_1')f+(q_1',q_2)f.
\end{flalign*}

We will therefore assume that $a_1=a_2=1$ and thus $p_0<p_1<p_2$. We will now represent all terms in the claimed equality in terms of a single cube $t=f[p_2]$ in the tessellation $f$.

\begin{flalign*}
&(p_0,p_1)f/\wgt_{n}(p_2)=\wgt_{n}(p_1-p_2)\lbl(f,E_{p_2}(\faceat(-e_{i_2},-e_{i_1}))) \\
&=n[i_2]\faceat(-e_{i_2},-e_{i_1})t=n[i_2]t[-e_{i_2},-e_{i_2}-e_{i_1}] \\
&(p_1,p_2)f/\wgt_{n}(p_2)=\wgt_{n}(p_2-p_2)\lbl(f,E_{p_2}(\faceat(\cvec{0},-e_{i_2}))) \\
&=\faceat(\cvec{0},-e_{i_2})t=t[\cvec{0},-e_{i_2}] \\
&(p_0,p_1')f/\wgt_{n}(p_2)=\wgt_{n}(p_1'-p_2)\lbl(f,E_{p_2}(\faceat(-e_{i_1},-e_{i_2})))\\
&=n[i_1]\faceat(-e_{i_1},-e_{i_2})t=n[i_1]t[-e_{i_1},-e_{i_1}-e_{i_2}] \\
&(p_1',p_2)f/\wgt_{n}(p_2)=\wgt_{n}(p_2-p_2)\lbl(f,E_{p_2}(\faceat(\cvec{0},-e_{i_1})))\\
&=\faceat(0,-e_{i_1})t=t[\cvec{0},-e_1]
\end{flalign*}

From this we see that

\begin{flalign*}
&((p_0,p_1)f+(p_1,p_2)f)/\wgt_{n}(p_2) \\
&=n[i_2]t[-e_{i_2},-e_{i_2}-e_{i_1}]+t[\cvec{0},-e_{i_2}]\overset{L.~\ref{cubePath}}{=}t[\cvec{0},-e_{i_1}-e_{i_2}]\mbox{ and}\\
&((p_0,p_1')f+(p_1',p_2)f)/\wgt_{n}(p_2) \\
&=n[i_1]t[-e_{i_1},-e_{i_1}-e_{i_2}]+t[\cvec{0},-e_{i_1}]\overset{L.~\ref{cubePath}}{=}t[\cvec{0},-e_{i_1}-e_{i_2}].
\end{flalign*}
\end{proof}

For a path $P=(P_i)_{i=1}^k$ let $E_P(i)=a_i e_{m(i)}$ for $1\leq i<k$. We say that an integer $i$ with $1\leq i<k-1$ is a \emph{return} if $m(i)=m(i+1)$ and $a_i\neq a_{i+1}$. We say that $i$ is an \emph{inversion} if $m(i)>m(i+i)$. The abelianization of $P$ is the unique path $\abel(P)=(P_i')_{i=1}^{k'}$ with $E_{P'}(i)=a_i' e_{m'(i)}$ for $1\leq i<k'$ that satisfies

\begin{itemize}
\item $P_1'=P_1$ and $P_k'=P_k$ ($P'$ and $P$ have the same start and end points)
\item if $1\leq i<k'-1$, then $m'(i)\leq m'(i+1)$ ($P'$ has no inversions)
\item if $1\leq i<k'-1$ and $m'(i)=m'(i+1)$, then $a_{i}'=a_{i+1}'$ ($P'$ has no returns)
\end{itemize}

We will use the abelianization to show that the integral of a tiling over a path depends only on the choice of the start and end points of the path. For the proof of the following lemma, let $\ell(P)$ be the natural number whose base-$d$ expansion is $m(1)\cdots m(k)$.

\begin{lemma}
For any path $P$ and any valid tiling $f\in X_{n}$ it holds that $Pf=\abel(P)f$.
\end{lemma}
\begin{proof}
We will manipulate any given path $P\neq\abel(P)$ into another path $P'$ which satisfies $\ell(P')<\ell(P)$. We also make sure that $P'$ has the same start and end points as $P$ (so $\abel(P')=\abel(P)$) and that $P'f=Pf$. Since $\ell(P)$ is a natural number, this argument can be repeated only finitely many times and it will eventually yield the path $\abel(P)$. Assume therefore that $P\neq\abel(P)$ and let $E_P(i)=a_i e_{m(i)}$ for $1\leq i<k$. There are two possibilities.

Case 1, a return: $m(i)=m(i+1)$ and $a_i\neq a_{i+1}$ for some $1\leq i<k-1$. Let $P'=(P_1,\dots,P_i,P_{i+3},\dots,P_k)$. Then, using the fact that $P_{i+2}=P_i$, we see from Lemma~\ref{cancel} that $(P_i,P_{i+1})f+(P_{i+1},P_{i+2})f=0$ and therefore

\begin{flalign*}
Pf=&\sum_{j=1}^{i-1}(P_j,P_{j+1})f+(P_{i+2},P_{i+3})f+\sum_{j=i+3}^{k-1}(P_j,P_{j+1})f \\
=&\sum_{j=1}^{i-1}(P_j,P_{j+1})f+(P_i,P_{i+3})f+\sum_{j=i+3}^{k-1}(P_j,P_{j+1})=P'f.
\end{flalign*}

The path $P'$ has the same start and end points as $P$ and it is shorter than $P$, so $\ell(P')<\ell(P)$.

Case 2, an inversion: $m(i)>m(i+1)$ for some $1\leq i<k-1$. Let $P_{i+1}'=P_i+(P_{i+2}-P_{i+1})$ and let $P'=(P_1,\dots,P_{i},P_{i+1}',P_{i+2},\dots,P_k)$. It follows from Lemma~\ref{switch} that $(P_i,P_{i+1})f+(P_{i+1},P_{i+2})f=(P_i,P_{i+1}')f=(P_{i+1}',P_{i+2})f$ and therefore $Pf=P'f$. Letting $E_{P'}(j)=b_j e_{m'(i)}$ for $1\leq j<k$, we check that
\[b_i e_{m'(i)}=P_{i+1}'-P_i=P_{i+2}-P_{i+1}=a_{i+1}e_{m(i+1)}\]
and thus $m'(i)=m(i+1)<m(i)$, so $\ell(P')<\ell(P)$.
\end{proof}

\begin{theorem}\label{indPath}
If two paths $P$, $P'$ have the same start and end points and $f\in X_{n}$, then $Pf=P'f$. In particular, if $P$ is a cycle, then $Pf=0$.
\end{theorem}
\begin{proof}
This follows from the previous lemma, because $\abel(P)=\abel(P')$.
\end{proof}

\begin{remark}
The inspiration for the terminology of ``path integral'' comes from line integrals of holomorphic functions, and Theorem~\ref{indPath} is analogous to the fact that the line integral of a holomorphic function over a closed curve is equal to zero. In the case of meromorphic functions, if the line integral over a closed curve is not zero, its interior necessarily contains a non-removable singularity. Similarly, if we are given only a partial tiling of the plane and the path integral around a non-tiled part is not zero, then the interior of the path cannot be completed into a valid tiling: see Figure~\ref{path} for an example.

The situation here is reminiscent of the tiling group approach of~\cite{Thu90} for determining whether a collection of polygonal tiles can tile a finite region of the plane. In that approach, the tiles can be interpreted as describing relators in a group, and if the group element describing the boundary of the region is not the identity element, then tiling the region is not possible.
\end{remark}

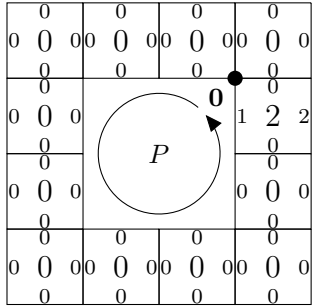
\begin{figure}[ht]
\centering
\begin{tikzpicture}[scale=1]

\tile[0,2,5](-3,0);\tile[0,2,5](-2,0);\tile[0,2,5](-1,0);\tile[0,2,5](0,0);
\tile[0,2,5](-3,-1);                                      \tile[2,2,5](0,-1);
\tile[0,2,5](-3,-2);                                      \tile[0,2,5](0,-2);
\tile[0,2,5](-3,-3);\tile[0,2,5](-2,-3);\tile[0,2,5](-1,-3);\tile[0,2,5](0,-3);

\fill   (-1,-1) circle (1mm);
\node at (-1.25,-1.25) {$\cvec{0}$};
\node at (-2,-2) {$P$};
\draw[-{triangle 45}] (-2,-2) ++ (50: 0.8) arc[radius=0.8,start angle=50,delta angle=350];

\node at (-7,-4.5) [anchor=west] {$P=(\cvec{0},-e_1,-2e_1,-2e_1-e_2,-2e_1-2e_2,-e_1-2e_2,-2e_2,-e_2,\cvec{0})$};
\node at (-7,-5.1) [anchor=west] {$Pf=0+0+0+0+0+0+0+1=1$};

\end{tikzpicture}
\caption{A partial valid tiling $f$ using $T_{(2,5)}$. There is no way to complete the tiling, because the path integral around the non-tiled part is not zero.}
\label{path}
\end{figure}

Now, given a tessellation $f\in X_{n}$ and a pair of vectors $p,p'\in\Z^d$, define
\[(p,p')f=Pf\]
for any path $P$ going from $p$ to $p'$: by Theorem~\ref{indPath} the choice of the path $P$ does not matter. For the directed pair $(p,p')$ we also define the (directed) label
\[\lbl(f,(p,p'))=(p,p')f/\wgt_{n}(p').\]
The directed and undirected labels agree in the sense that $\lbl(f,(z-e_i,z))=\lbl(f,\{z-e_i,z\})$ for $z\in\Z^d$ and $1\leq i\leq d$. As one would hope, the labels do not change when the tessellation is shifted.

\begin{lemma}\label{shiftlabel}
Let $f\in X_{n}$ and $p,p'\in\Z^d$. For any $v\in\Z^d$ it holds that 
\[\lbl(f,(p+v,p'+v))=\lbl(\sigma_v(f),(p,p')).\]
\end{lemma}
\begin{proof}
\begin{flalign*}
&\lbl(f,(p+v,p'+v))=(p+v,p'+v)f/\wgt_{n}(p'+v) \\
\overset{L.~\ref{pathshift}}&{=}\wgt_{n}(v)(p,p')\sigma_v(f)/\wgt_{n}(p'+v)=(p,p')\sigma_v(f)/\wgt_{n}(p') \\
&=\lbl(\sigma_v(f),(p,p')).
\end{flalign*}
\end{proof}

The labeling map inherits a summation property from the summation property of paths. The next lemma associates to any label in a tiling a mixed base representation of sorts, where the digits also come from labels in the tiling. In a special case the base consists of powers of a single number $m(n,v)\in\Q$, which is however not necessarily an integer.

\begin{lemma}\label{pathsum}
Let $p_0,p_1,\dots,p_k\in\Z^d$ and $f\in X_{n}$. Then
\[\lbl(f,(p_k,p_0))=\sum_{i=0}^{k-1} m(n,p_{i}-p_0)\lbl(f,(p_{i+1},p_{i})).\]
In particular, for $v\in\Z^d$ and $k\in\N$ it holds that
\[\lbl(f,(kv+p_0,p_0))=\sum_{i=0}^{k-1} m(n,v)^i\lbl(f,((i+1)v+p_0,iv+p_0)).\]
\end{lemma}
\begin{proof}
\begin{flalign*}
&\lbl(f,(p_k,p_0))=(p_k,p_0)f/\wgt_{n}(p_0)=\sum_{i=0}^{k-1}(p_{i+1},p_{i})f/\wgt_{n}(p_0) \\
&=\sum_{i=0}^{k-1} \wgt_{n}(p_i)\lbl(f,(p_{i+1},p_{i}))/\wgt_{n}(p_0)=\sum_{i=0}^{k-1} \wgt_{n}(p_i-p_0)\lbl(f,(p_{i+1},p_{i}))
\end{flalign*}
For the latter part, choose $p_i=iv+p_0$ and note that
\[m(n,p_i-p_0)=m(n,iv)=\prod_{j=1}^d n[j]^{-iv[j]}=m(n,v)^i.\]
\end{proof}

The base $m(n,v)$ of Lemma~\ref{pathsum} is a natural number at least when $v\leq\cvec{0}$. In that case the ``digits'' $\lbl(f,((i+1)v+p_0,iv+p_0))$ are actually just usual base-$m(n,v)$ digits, which follows from the next lemma.

\begin{lemma}\label{labelbound}
If $p\leq p'\in\Z^d$ and $f\in X_{n}$, then $\lbl(f,(p,p'))\in\N$ and $0\leq\lbl(f,(p,p'))<m(n,p-p')$.
\end{lemma}
\begin{proof}
By Lemma~\ref{shiftlabel} it is sufficient to show that if $p\leq\cvec{0}$, then $\lbl(f,(p,\cvec{0}))\in\N$ and $\lbl(f,(p,\cvec{0}))<m(n,p)$. The proof is by induction on the distance of $p$ from the origin, the base case $p=\cvec{0}$ is simple. Assume then that the claim holds for some $p\leq\cvec{0}$ and consider the vector $p-e_i$ for some $1\leq i\leq d$. Then
\begin{flalign*}
&\lbl(f,(p-e_i,\cvec{0}))=(p-e_i,\cvec{0})f=(p-e_i,p)f+(p,\cvec{0})f \\
&=\underbrace{\wgt_{n}(p)}_{\in\N}\underbrace{\lbl(f,\{p-e_i,p\})}_{\in\N}+\underbrace{\lbl(f,(p,\cvec{0}))}_{\in\N}<m(n,p)\faceat(\cvec{0},-e_i)f[p]+m(n,p) \\
&=m(n,p)(\base_{n}(f[p],(\cvec{0},-e_i))[2]+1)\leq m(n,p) n[i]=m(n,p-e_i),
\end{flalign*}
where we use the fact that $\base_{n}(f[p],(\cvec{0},-e_i))[2]<m(n,-e_i)/m(n,\cvec{0})=n[i]$.
\end{proof}

\begin{lemma}\label{baselabel}
Let $p_k\leq\cdots\leq p_0$ and $f\in X_{n}$. Then for $1\leq i\leq k'\leq k$,
\[\base_{n}(\lbl(f,(p_k,p_0)),(p_j-p_0)_{j=1}^{k'})[i]=\lbl(f,(p_i,p_{i-1})).\]
\end{lemma}
\begin{proof}
By Lemma~\ref{pathsum}
\begin{flalign*}
\lbl(f,(p_k,p_0))=\sum_{i=0}^{k-1} m(n,p_{i}-p_0)\lbl(f,(p_{i+1},p_{i})).
\end{flalign*}
From this the claim follows for $k'=k$ by the definition of a mixed base representation, because by Lemma~\ref{labelbound} $0\leq\lbl(f,(p_{i+1},p_{i}))<m(n,p_{i+1}-p_{i})=m(n,p_{i+1}-p_0)/m(n,p_{i}-p_0)$ for $1\leq i\leq k$. The claim for general $k'$ follows from the case $k'=k$ by Lemma~\ref{endscope}.
\end{proof}

\begin{lemma}\label{diaglabelvalue}
Let $f\in X_{n}$ and $z\in\Z^d$. Then $\val_{n}(f[z])=\lbl(f,(z-\cvec{1},z))$.
\end{lemma}
\begin{proof}
Up to shifting $f$ we may assume that $z=\cvec{0}$. Let $t=f[\cvec{0}]$. For $0\leq i\leq d$ let $e_i'=\sum_{j=1}^i e_j$. Then $e_i'$ form a binary directive sequence. Using the definition of the undirected label of a single edge we see that
\[\lbl(f,\{e_{i+1}',e_i'\})=\faceat(e_i',e_{i+1})f[\cvec{0}]=t[e_i',e_{i+1}'].\]

Using this we compute that 
\begin{flalign*}
\lbl(f,(-\cvec{1},\cvec{0}))\overset{L.~\ref{pathsum}}&{=}\sum_{i=0}^{d-1}m(n,e_i')\lbl(f,\{e_{i+1}',e_i'\}) \\
&=\sum_{i=0}^{d-1}m(n,e_i')t[e_i',e_{i+1}']\overset{L.~\ref{cubePath}}{=}t[\cvec{0},\cvec{1}]=\val_{n}(t).
\end{flalign*}
\end{proof}

For a face $s=\faceat(p,v)$ and $z\in\Z^d$ denote $\dedge_z(s)=(z+p+v,z+p)$ (this generalizes an earlier definition where $s$ was an edge). We generalize the defining formula for the label of an edge in the following lemma.

\begin{lemma}\label{facelabelvalue}
Let $f\in X_{n}$.  For $s=\faceat(p,v)$ and $z\in\Z^d$ it holds that
\[\lbl(f,\dedge_z(s))=sf[z].\]
\end{lemma}
\begin{proof}
Up to shifting $f$ we may assume that $z=\cvec{0}$. Let $\cube_{n}(a)=f[\cvec{0}]$. Then, by applying Lemma~\ref{baselabel} with the choice $(p_0,p_1,p_2,p_3)=(\cvec{0},p,p+v,\cvec{1})$,
\begin{flalign*}
\lbl(f,(\dedge_{\cvec{0}}(s)))&=\lbl(f,(p+v,p))\overset{L.~\ref{baselabel}}{=}\base_{n}(\lbl(f,(-\cvec{1},\cvec{0})),(p,p+v))[2] \\
\overset{L.~\ref{diaglabelvalue}}&{=}\base_{n}(a,(p,p+v))[2]=\cube_n(a)[p,p+v]=sf[\cvec{0}].
\end{flalign*}
\end{proof}

As we shall see, by Lemma~\ref{diaglabelvalue} labels along the direction of $\cvec{1}$ can be found by reading the values of a sequence of cubes and interpreting this sequence of values as the representation of a number in base $N=\prod_{i=1}^d n[i]$. We actually prove a more general statement by using Lemma~\ref{facelabelvalue}. 

\begin{lemma}\label{labelIsBaseexp}
Let $p_k\leq\cdots\leq p_0$ with $p_{i+1}-p_i\geq-\cvec{1}$ for $0\leq i<k$ and $f\in X_{n}$. Then
\[\lbl(f,(p_k,p_0))=\sum_{i=0}^{k-1} m(n,p_{i}-p_0)\faceat(\cvec{0},p_{i+1}-p_i)f[p_i].\]
In particular, for any $k\in\N$ and $v\in\Z^d$,
\[\lbl(f,(v-\cvec{k},v))=\sum_{i=0}^{k-1} N^i\val_n(f[v+p_i]).\]
\end{lemma}
\begin{proof}
The first claim follows from Lemmas~\ref{pathsum} and~\ref{facelabelvalue}. The second claim follows from the first with the choice $p_i=v-\cvec{i}$ for $0\leq i\leq k$.
\end{proof}

We conclude this subsection with some justification for why the elements of $T_n$ are called multiplication cubes.

\begin{proposition}\label{shiftMulProp}
Let $p, p'\in\Z^d$, $v\in\Z^d$ and $f\in X_n$. Then
\[\lbl(f,(p+v,p'+v)) = m(n,-v)\lbl(f,(p,p'))+m(n,-p'-v)((p+v,p)f+(p',p'+v)f).\]
\end{proposition}
\begin{proof}
By Theorem~\ref{indPath} $(p+v,p'+v)f=(p+v,p)f+(p,p')f+(p',p'+v)f$, and by replacing path integrals with labels we find
\[m(n,p'+v)\lbl(f,(p+v,p'+v))=(p+v,p)f+m(n,p')\lbl(f,(p,p'))+(p',p'+v)f.\]
The claim follows by dividing both sides of the equation with $m(n,p'+v)$.
\end{proof}

The message of this proposition is that moving in $\Z^d$ by a vector $v$ changes the values of the labeling map by the factor of $m(n,-v)$, possibly up to an error term. In particular, moving to the direction of a basis vector $e_i$ corresponds to multiplication by $n[i]$ in this sense. Sometimes the error term can be made equal to zero. We also remark that another type of tile set performing multiplication in a similar sense has appeared in~\cite{Kari96b}.

\begin{example}\label{powersof2}
We may now explain the powers of two in Figure~\ref{tiling}. Observe that in the figure $(-\cvec{3}+e_1,-\cvec{3})f=(\cvec{0},e_1)f=0$. Then by applying the previous proposition for $p=-\cvec{3}$, $p'=\cvec{0}$ and $v=e_1$ it follows that
\[\lbl(f,(-\cvec{3}+e_1,e_1))=m((2,5),-e_1)\lbl(f,(-\cvec{3},\cvec{0}))=2\lbl(f,(-\cvec{3},\cvec{0})).\]
Alternatively, by using the latter part of Lemma~\ref{labelIsBaseexp} the values $\lbl(f,(-\cvec{3},\cvec{0}))=64$ and $\lbl(f,(-\cvec{3}+e_1,e_1))=128$ can be read directly from the figure. Similar arguments connect the other powers of two appearing in the figure with each other. 
\end{example}

Next we show that Proposition~\ref{shiftMulProp} can be extended to path integrals over some infinite paths so that the error terms vanish under very natural assumptions, thus connecting multiplication operations to moving around tessellations in an even more satisfactory sense. Given vectors $p,v\in\Z^d$ with $v\leq\cvec{0}$ we consider the infinite path going through $p$ along $v$, and given $f\in X_n$ we denote

\[\fractional_{p,v}(f)=\lim_{i\to\infty}(p,p-iv)f\mbox{ and }\integ_{p,v}(f)=\lim_{i\to\infty}(p+iv,p)f.\]
Alternatively, we can rewrite
\begin{flalign*}
&\fractional_{p,v}(f)=\sum_{i=0}^\infty(p-i v,p-(i+1)v)f  \\
&=\wgt_n(p)\sum_{i=0}^\infty\wgt_n(v)^{-(i+1)}\lbl(f,(p-i v,p-(i+1)v))\quad\mbox{and} \\
&\integ_{p,v}(f)=\sum_{i=0}^\infty(p+(i+1)v,p+i v)f \\
&=\wgt_n(p)\sum_{i=0}^\infty\wgt_n(v)^{i}\lbl(f,(p+(i+1)v,p+iv)),
\end{flalign*}
where $\lbl(f,(p-i v,p-(i+1)v))$ and $\lbl(f,(p+(i+1)v,p+iv))$ are natural numbers less than $\wgt_n(v)$ by Lemma~\ref{labelbound}. The names of these quantities come from the fact that these are essentially (up to ignoring the factor $\wgt_n(p)$) the fractional parts and integral parts of some quantity represented in base $\wgt_n(p)$: $\fractional_{p,v}(f)/\wgt_n(p)\in[0,1]$ and $\integ_{p,v}(f)/\wgt_n(p)\in\N$ if it is finite.

We write
\begin{flalign*}
&\real_{p,v}(f)=\integ_{p,v}(f)+\fractional_{p,v}(f)=\lim_{i\to\infty}(p+iv,p-iv)f \\
&=\wgt_n(p)\sum_{i=-\infty}^\infty\wgt_n(v)^{i}\lbl(f,(p+(i+1)v,p+iv)).
\end{flalign*}
In the special case $p=\cvec{0}$ we may omit the subscript $p$, and in this case the sum above yields a base-$\wgt_n(p)$ representation for $\real_v(f)$. If additionally $v=-\cvec{1}$, we may also omit the subscript $v$ and say that the tessellation $f$ represents the number $\real(f)$, but in Proposition~\ref{freepv} it will turn out that the choices of $p$ and $v$ do not matter much. By Lemma~\ref{diaglabelvalue} the number $\real(f)$ can be read directly from the cubes on the main diagonal of the tessellation $f$:
\[\real(f)=\sum_{i=-\infty}^\infty\wgt_n(-\cvec{1})^{i}\val_n(f[-\cvec{i}])=\sum_{i=-\infty}^\infty N^i\val_n(f[-\cvec{i}]).\]

If $\wgt_n(v)>1$, then in the sum definiting $\integ_{p,v}(f)$ the part $\wgt_n(v)^{i}$ tends to infinity and so the finiteness of the sum is equivalent to the equality $\lbl(f,(p+(i+1)v,p+iv))=0$ (or equivalently $(p+(i+1)v,p+i v)f=0$) holding for all sufficiently large $i$. In fact, finiteness of $\integ_{p,v}(f)$ also implies that many other labels are equal to zero.

\begin{lemma}\label{intCutoff}
The implications~\ref{first}$\implies$\ref{second}$\implies$\ref{third} hold for the statements
\begin{enumerate}
\item\label{first}$\integ_{p,v}(f)$ is finite for some $p,v\in\Z^d$ with $v\ll \cvec{0}$ and $\wgt_n(v)>1$,
\item\label{second}$\lbl(f,(p_2,p_1))=(p_2,p_1)f=0$ whenever $p_2\leq p_1$ and $\wgt_n(p_1)$ is sufficiently large,
\item\label{third}$\integ_{p,v}(f)$ is finite for all $p,v\in\Z^d$ with $v\leq\cvec{0}$ and $\wgt_n(v)>1$.
\end{enumerate}
In particular, finiteness of $\integ_{p,v}(f)$ does not depend on the choice of $p$ and $v$ among vectors such that $v\ll\cvec{0}$ and $\wgt_n(v)>1$.
\end{lemma}
\begin{proof}
To show the implication~\ref{first}$\implies$\ref{second}, assume that $\integ_{p,v}(f)$ is finite for some $v\ll \cvec{0}$ and $\wgt_n(v)>1$ and consider any $p_2\leq p_1$ such that $\wgt_n(p_1)>\real_{p,v}(f)$. Since $v\ll\cvec{0}$, we may fix some $I\in\N$ such that $p_1\leq p-I v$ and $p+I v\leq p_2$. If it were the case that $\lbl(f,(p_2,p_1))>0$, then
\begin{flalign*}
&\real_{p,v}(f)\geq(p+Iv,p-I v)f=(p+Iv,p_2)f+(p_2,p_1)f+(p_1,p-Iv)f \\
&\geq (p_2,p_1)f=\wgt_n(p_1)\lbl(f,(p_2,p_1))>\real_{p,v}(f)\lbl(f,(p_2,p_1))
\end{flalign*}
and $\lbl(f,(p_2,p_1))>0$ implies that $\real_{p,v}(f)>0$, but dividing by $\real_{p,v}(f)$ in the inequalities yields $1>\lbl(f,(p_2,p_1))$, a contradiction.

To show the implication~\ref{second}$\implies$\ref{third}, assume that $\lbl(f,(p_2,p_1))=0$ whenever $p_2\leq p_1$ and $\wgt_n(p_1)$ is sufficiently large. The number $\wgt_n(p+i v)$ tends to infinity as $i$ tends to infinity, so by the assumption $\lbl(f,(p+(i+1)v,p+iv))=0$ for sufficiently large $i$ and thus $\integ_v(f)$ is finite.
\end{proof}

\begin{lemma}
For $p,q,v,w\in\Z^d$ satisfying $v,w\leq\cvec{0}$ and $\wgt_n(v),\wgt_n(w)>1$ and for $f\in X_n$ it holds that $\fractional_{p,v}(f)=(p,q)f+\fractional_{q,w}(f)$.
\end{lemma}
\begin{proof}
It suffices to prove that $\fractional_{-\cvec{k}}(f)=(\cvec{0},p)f+\fractional_{p,v}(f)$ for sufficiently large $k\in\N$, because then
\[(\cvec{0},p)f+\fractional_{p,v}(f)=\fractional_{-\cvec{k}}(f)=(\cvec{0},q)f+\fractional_{q,w}(f).\]
Let therefore $k\in\N$ be so large that $-\cvec{k}=v+v'$ for some $v'\ll\cvec{0}$. For sufficiently large $i\in\N$ it holds that $p-iv\leq -iv-iv'$, and so an application of Lemma~\ref{labelbound} at the position indicated below shows that the following expression is nonnegative and simultaneously bounds it from above:
\begin{flalign*}
&\fractional_{-\cvec{k}}(f)-\fractional_{p,v}(f)-(\cvec{0},p)f=\lim_{i\to\infty}(\cvec{0},i\cvec{k})f-(p,p-iv)f-(\cvec{0},p)f \\
&=\lim_{i\to\infty}(\cvec{0},i\cvec{k})f-(\cvec{0},p-iv)f=\lim_{i\to\infty}(p-iv,i\cvec{k})f \\
&=\lim_{i\to\infty}(p-iv,-iv-iv')f=\lim_{i\to\infty}\lbl(f,(p-iv,-iv-iv'))\wgt_n(v+v')^{-i} \\
\overset{L.~\ref{labelbound}}&{\leq}\lim_{i\to\infty}\wgt_n(p+iv')\wgt_n(v+v')^{-i}=\lim_{i\to\infty}\wgt_n(p)\wgt_n(v)^{-i}=0.
\end{flalign*}
\end{proof}

\begin{lemma}
Let $p,q,v,w\in\Z^d$ satisfy $v\ll\cvec{0}$, $w\leq\cvec{0}$ and $\wgt_n(v),\wgt_n(w)>1$ and let $f\in X_n$. If $\integ_{p,v}(f)$ is finite, it holds that $\integ_{p,v}(f)+(p,q)f=\integ_{q,w}(f)$.
\end{lemma}
\begin{proof}
From Lemma~\ref{intCutoff} it follows that $\integ_{q,w}(f)$ is also finite. By using Lemma~\ref{intCutoff} choose a sufficiently large $I\in\N$ so that $(p',p+Iv)f=0$ and $(q',q+Iw)f=0$ whenever $p'\leq p+Iv$ and $q'\leq q+Iw$. In particular $\integ_{p,v}(f)=(p+Iv,p)f$ and $\integ_{q,w}(f)=(q+Iw,q)f$. Fix some $r\in\Z^d$ such that $r\leq p+Iv$ and $r\leq p+Iw$. Then
\begin{flalign*}
&\integ_{p,v}(f)+(p,q)f=(r,p+Iv)f+(p+Iv,p)f+(p,q)f=(r,q)f \\
&=(r,q+Iw)f+(q+Iw,q)f=\integ_{q,w}(f).
\end{flalign*}
\end{proof}

\begin{proposition}\label{freepv}
Let $p,q,v,w\in\Z^d$ be such that $v\ll\cvec{0}$, $w\leq\cvec{0}$ and $\wgt_n(v),\wgt_n(w)>1$ and let $f\in X_n$. If $\real_{p,v}(f)$ is finite, it holds that $\real_{p,v}(f)=\real_{q,w}(f)$.
\end{proposition}
\begin{proof}
Combine the two previous lemmas to see that 
\[\fractional_{p,v}(f)+\integ_{p,v}(f)+(p,q)f=(p,q)f+\fractional_{q,w}(f)+\integ_{q,w}(f)\]
and subtract $(p,q)f$ from both sides of the equality.
\end{proof}

We are now ready to present the analogue of Proposition~\ref{shiftMulProp} for infinite paths. The following proposition shows that shifting a tessellation by $\sigma_v$ multiplies the real number it represents by $\wgt_n(-v)$ without additional error terms.

\begin{proposition}\label{infShiftMul}
For $f\in X_n$ such that $\real(f)$ is finite and $v\in\Z^d$ it holds that $\real(\sigma_v(f))=\wgt_n(-v)\real(f)$.
\end{proposition}
\begin{proof}
We compute
\begin{flalign*}
&\real(\sigma_v(f))\overset{P.~\ref{freepv}}{=}\real_{-v,-\cvec{1}}(\sigma_v(f))=\lim_{i\to\infty}(-v-\cvec{i},-v+\cvec{i})\sigma_v(f)\\
&\overset{L.~\ref{pathshift}}{=}\wgt_n(-v)\lim_{i\to\infty}(-\cvec{i},\cvec{i})f=\wgt_n(-v)\real(f).
\end{flalign*}
\end{proof}

\subsection{Macrotiles and microtiles}\label{macMicSubSect}

In Figure~\ref{macrotiling} tiles of $T_{(2,5)}$ are grouped into $2\times 2$ squares, values in the centers of the new squares are given by computing the labels from the bottom left to the top right in the original $2\times 2$ squares, and labels for the edges of the new squares are given by computing the labels of the boundaries of the original $2\times 2$ squares. It turns out that the resulting squares are tiles of $T_{(4,25)}$. Grouping these new tiles again into $2\times 2$ squares yields tiles of $T_{(16,625)}$. In this subsection we show how partitioning a valid tessellation into larger squares, and even into more general (multidimensional) parallelepipeds yields new multiplication cubes.

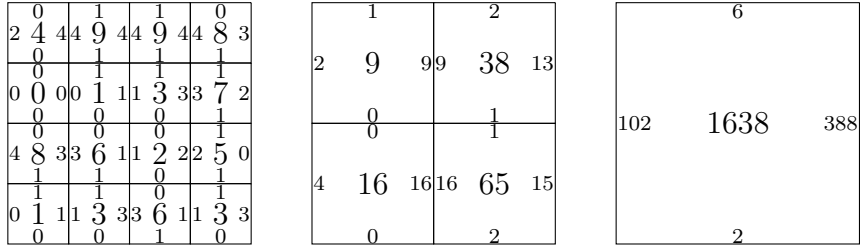
\begin{figure}[ht]
\centering
\begin{tikzpicture}[scale=0.8]

\tile[4,2,5](0,0);\tile[9,2,5](1,0);\tile[9,2,5](2,0);\tile[8,2,5](3,0);
\tile[0,2,5](0,-1);\tile[1,2,5](1,-1);\tile[3,2,5](2,-1);\tile[7,2,5](3,-1);
\tile[8,2,5](0,-2);\tile[6,2,5](1,-2);\tile[2,2,5](2,-2);\tile[5,2,5](3,-2);
\tile[1,2,5](0,-3);\tile[3,2,5](1,-3);\tile[6,2,5](2,-3);\tile[3,2,5](3,-3);

\tilescale[9,4,25,2](6,0);\tilescale[38,4,25,2](8,0);
\tilescale[16,4,25,2](6,-2);\tilescale[65,4,25,2](8,-2);

\tilescale[1638,16,625,4](13,0);

\end{tikzpicture}
\caption{Tiles from $T_{(2,5)}$ (left) grouped into larger macrotiles of $T_{(4,25)}$ (middle) and $T_{(16,625)}$ (right).}
\label{macrotiling}
\end{figure}

Throughout let $n$ be a prebasis of dimension $d$ and let $A$ be a $d\times d'$ natural number matrix. We interpret the columns $Ae_i$ of $A$ as the generating vectors of the parallelepiped which will yield the new multiplication cubes, e.g. in the case of Figure~\ref{macrotiling} we would choose $A=\diag(2,2)$. We denote $n^A=m(n,(-Ae_i)_{i=1}^{d'})$, in particular $n^I=n$ for the identity matrix $I$. We define the $A$-macrotile map $\macro_{A,n}:X_{n}\to X_{n^{A}}$ (usually written just $\macro_A$) by
\[\macro_{A,n}(f)[v]=\cube_{n^A}((-A\cvec{1},\cvec{0})\sigma_{Av}(f))=\cube_{n^A}(\lbl(f,(A(v-\cvec{1}),Av)))\]
for $f\in X_{n}$, $v\in\Z^{d'}$. One needs to check that $\macro_A(f)$ really belongs to $X_{n^A}$. It is simple to verify that $\lbl(f,(A(v-\cvec{1}),Av))\in\digs_{N'}$ with 
\[N'=\prod_{i=1}^{d'} n^A[i]=\prod_{i=1}^{d'} m(n,-Ae_i)=m(n,-A\cvec{1}),\]
because $\lbl(f,(A(v-\cvec{1}),Av))<m(n,-A\cvec{1})=N'$ by Lemma~\ref{labelbound}.

In Theorem~\ref{macroMatch} it will turn out that labels in the original tiling correspond to the labels of faces in the macrotiling. From this it follows that the hyperfaces of neighboring cubes match in $\macro_A(f)$ and that $\macro_A(f)\in X_{n^A}$.

Let us observe that
\[m(n^A,v_i)=\prod_{j=1}^{d'}m(n,-Ae_j)^{-v_i[j]}=\prod_{j=1}^{d'}m(n,A(v_i[j]e_j))=m(n,Av_i)\]
and
\[\base_{n^A}(a,(v_i))=\base(a,m(n^A,(v_i)))=\base(a,m(n,(Av_i)))=\base_{n}(a,(Av_i)).\]

\begin{theorem}\label{macroMatch}
Let $f\in X_{n}$, $z\in\Z^{d'}$ and $t=\macro_A(f)[z]$. For $s=\faceat(p,u)$ with $p,u\in\Z^{d'}$ it holds that $st=\lbl(f,(A(z+p+u),A(z+p)))$. In particular, $\topi_i(\macro_A(f)[z])=\boti_i(\macro_A(f)[z+e_i])$ for any $1\leq i\leq d'$.
\end{theorem}
\begin{proof}
By applying Lemma~\ref{baselabel} with the choice 
\[(p_0,p_1,p_2,p_3)=(Az,A(z+p),A(z+p+u),A(z-\cvec{1}))\]
we have
\begin{flalign*}
st&=s\cube_{n^A}(\lbl(f,(A(z-\cvec{1}),Az))) \\
&=\base_{n^A}(\lbl(f,(A(z-\cvec{1}),Az),(p,p+u))[2] \\
&=\base_{n}(\lbl(f,(A(z-\cvec{1}),Az)),(Ap,A(p+u)))[2] \\
\overset{L.~\ref{baselabel}}&{=}\lbl(f,(A(z+p+u),A(z+p))),
\end{flalign*}
which proves the first claim. To prove the second claim, we apply the first claim twice:
\begin{flalign*}
&\topi_i(\macro_A(f)[z])=\faceat(\cvec{0},-\cvec{1}+e_i)\macro_A(f)[z]=\lbl(f,(A(z-\cvec{1}+e_i),Az)) \\
&=\faceat(-e_i,-\cvec{1}+e_i)\macro_k(f)[z+e_i]=\boti_i(\macro_A(f)[z+e_i]).
\end{flalign*}
\end{proof}

\begin{remark}\label{facesAreMulcubes}
Consider the special case where the columns of $A$ are the first $d'$ elementary basis vectors of $\Z^d$. Then for every $z'\in\Z^{d'}$ we have $Az'=z'\ins_{d'}\cvec{0}$, where $\cvec{0}\in\Z^{d-d'}$. Fix $z'\in\Z^{d'}$ and let $z=Az'$. Let $s'=\faceat(p',u')$ for $p',u'\in\Z^{d'}$, $p=Ap'$, $u=Au'$. Then $p,u\in\{-1,0\}$ are orthogonal, so we may define $s=\faceat(p,u)$ and if follows that
\begin{flalign*}
s'(\macro_A(f)[z'])\overset{T.~\ref{macroMatch}}&{=}\lbl(f,(A(z'+p'+u'),A(z'+p'))) \\
&=\lbl(f,\dedge_z(s))\overset{L.~\ref{facelabelvalue}}{=}sf[z].
\end{flalign*}
The meaning of this equality is that the tessellation $\macro_A(f)$ is formed by looking at a cut of the tessellation $f$ along the first $d'$ coordinate axes.

As a more concrete example, consider $f\in X_{(2,3,5)}$ and let the columns of $A$ be the basis vectors $e_1=(1,0,0)$ and $e_2=(0,1,0)$. For any $z'=(z_1,z_2)\in\Z^2$ and any face of the two-dimensional cube $s=\faceat((p_1,p_2),(u_1,u_2))$ one finds that 
\begin{flalign*}
&\faceat((p_1,p_2),(u_1,u_2))\macro_A(f)[(z_1,z_2)] \\
&=\faceat((p_1,p_2,0),(u_1,u_2,0))f[(z_1,z_2,0)].
\end{flalign*}
Assuming that the cube in Figure~\ref{cube} is equal to $f[(z_1,z_2,0)]$, then the values of all the hyperfaces of $\macro_A(f)[(z_1,z_2)]$ are equal to the values on the corresponding hyperfaces of $f[(z_1,z_2,0)]$ on the plane $z=0$. This explains why the top face and its adjacent edges in Figure~\ref{cube} yield a multiplication cube from $T_{(2,3)}$.
\end{remark}

\begin{lemma}\label{macroedgeSurj}
Let $f\in X_{n}$ and let $p,p'\in\Z^{d'}$. Then
\[\lbl(\macro_A(f),(p,p'))=\lbl(f,(Ap,Ap'))\quad\mbox{and}\quad(p,p')\macro_A(f)=(Ap,Ap')f.\]
\end{lemma}
\begin{proof}
The two equalities are equivalent, so it suffices to prove the first one. We first show for $z,v\in\Z^d$, $v\geq\cvec{0}$ that
\[\lbl(\macro_A(f),(z-v,z))=\lbl(f,(A(z-v),Az)).\]
In the special case $v=e_i$ for some $1\leq i\leq d'$, let $t=\macro_A(f)[z]$ and $s=\faceat(\cvec{0},-e_i)$. Then by Theorem~\ref{macroMatch}
\begin{flalign*}
\lbl(\macro_A(f),(z-e_i,z))=\lbl(\macro_A(f),\edge_z(s))=st=\lbl(f,(A(z-e_i),Az)).
\end{flalign*}
In general $v=\sum_{k=1}^m e_{i_k}$ for some $m\in\N$. Denote $v_j=\sum_{k=1}^j e_{i_k}$ for $0\leq j\leq m$.
\begin{flalign*}
&\lbl(\macro_A(f),(z-v,z))\overset{L.~\ref{pathsum}}{=}\sum_{i=0}^{m-1}m(n^A,-v_i)\lbl(\macro_A(f),(z-v_{i}-e_{i+1},z-v_i)) \\
&=\sum_{i=0}^{m-1}m(n,-Av_i)\lbl(f,(A(z-v_{i}-e_{i+1}),A(z-v_i))) \\
\overset{L.~\ref{pathsum}}&{=}\lbl(f,(A(z-v),Az).
\end{flalign*}

To prove the statement of the lemma, let $p''\in\Z^d$ such that $p''\geq p$ and $p''\geq p'$. The previous part is applied in the following with the choices $z=p''$ and $v=p''-p$ or $v=p''-p'$.
\begin{flalign*}
&\lbl(\macro_A(f),(p,p')) \\
\overset{L.~\ref{pathsum}}&{=}m(n^A,p''-p')\lbl(\macro_A(f),(p,p''))+\lbl(\macro_A(f),(p'',p')) \\
&=m(n,A(p''-p'))\lbl(f,(Ap,Ap''))+\lbl(f,(Ap'',Ap'))\overset{L.~\ref{pathsum}}{=}\lbl(f,(Ap,Ap')).
\end{flalign*}
\end{proof}

Now we can show that the real number represented by a tessellation is preserved by the macrotile map $\macro_A$. 

\begin{proposition}\label{surjtessreal}
Let $A$ be a $d\times d'$ natural number matrix such that $N'=\wgt_n(-A\cvec{1})>1$. For any $f\in X_n$ such that $\real(f)$ is finite it holds that $\real(\macro_A(f))=\real(f)$.
\end{proposition}
\begin{proof}
Finiteness of $\real(f)$ will be required for applying Lemma~\ref{freepv}. We compute
\begin{flalign*}
&\real(\macro_A(f))=\lim_{i\to\infty}(-\cvec{i},\cvec{i})\macro_A(f) \\
&\overset{L.~\ref{macroedgeSurj}}{=}\lim_{i\to\infty}(-iA\cvec{1},iA\cvec{1})f=\real_{-A\cvec{1}}(f)\overset{P.~\ref{freepv}}{=}\real(f).
\end{flalign*}
\end{proof}

In the following lemmas we show that the composition of macrotile maps corresponds to matrix multiplication.

\begin{lemma}\label{matrixComp}
Let $A$ be a $d\times d'$ and let $B$ be a $d'\times d''$ natural number matrix. Then $(n^A)^B=n^{AB}$.
\end{lemma}
\begin{proof}
\[(n^A)^B=m(n^A,(-Be_i)_{i=1}^{d''})=m(n,(-ABe_i)_{i=1}^{d''})=n^{AB}.\]
\end{proof}

\begin{lemma}\label{macroComp}
Let $A$ be a $d\times d'$ and let $B$ be a $d'\times d''$ natural number matrix. Then $\macro_{B,n^A}\circ\macro_{A,n}=\macro_{AB,n}$.
\end{lemma}
\begin{proof}
For $f\in X_{n}$ and $v\in\Z^{d''}$ we compute
\begin{flalign*}
&\macro_{B,n^A}(\macro_{A,n}(f))[v]\overset{L.~\ref{matrixComp}}{=}\cube_{n^{AB}}(\lbl(\macro_{A,n}(f),(B(v-\cvec{1}),Bv))) \\
\overset{L.~\ref{macroedgeSurj}}&{=}\cube_{n^{AB}}(\lbl(f,(AB(v-\cvec{1}),ABv)))=\macro_{AB,n}(f)[v]
\end{flalign*}
\end{proof}

Now let $A$ be a $d\times d'$ natural number matrix all of whose rows contain a positive integer. Then the $A$-macrotile map is going to have an inverse, the $A$-microtile map $\micro_{A,n}:X_{n^{A}}\to X_{n}$. By the assumption on $A$ for every element $x\in\Z^d$ there is a $z_1\in\Z^{d'}$ such that $Az_1\geq x$ and thus $x$ can be written in the form $x=Az_1-v$ with $v\in\Z^d$, $v\geq\cvec{0}$. Let $z_2\in\Z^{d'}$ satisfy $z_1\geq z_2$, $x-\cvec{1}\geq Az_2$, let $a=\lbl(f,(z_2,z_1))$ for $f\in X_{n^A}$ and define the map $\micro_{A,n}$ (usually written just $\micro_A$) by
\[\micro_{A,n}(f)[x]=\cube_{n}(\base_{n}(a,(-v,-v-\cvec{1}))[2]).\]
We need to verify that $\micro_A(f)[x]$ is well defined (which we will do in Lemma~\ref{microUnique}) and that $\micro_A(f)$ is indeed in $X_{n}$ (which we will do in Theorem~\ref{microMatch}). It is simple that $\base_{n}(a,(-v,-v-\cvec{1}))[2])\in \digs_N$ with $N=\prod_{i=1}^d n[i]$, since 
\[\base_{n}(a,(-v,-v-\cvec{1}))[2]<m(n,-v-\cvec{1})/m(n,-v)=m(n,-\cvec{1})=N.\]

\begin{lemma}\label{microUnique}
Let $f\in X_{n^A}$ and $x\in\Z^d$. Let $z_1,z_1'\in\Z^d$, $v,v'\in\Z^{d}$, $v,v'\geq\cvec{0}$ be such that $x=Az_1-v=Az_1'-v'$. Let $z_2,z_2'\in\Z^{d'}$ satisfy $z_1\geq z_2$, $z_1'\geq z_2'$, $x-\cvec{1}\geq Az_2,Az_2'$. Denote $a=\lbl(f,(z_2,z_1))$ and $a'=\lbl(f,(z_2',z_1'))$. Then
\[\base_{n}(a,(-v,-v-\cvec{1}))[2]=\base_{n}(a',(-v',-v'-\cvec{1}))[2].\]
\end{lemma}
\begin{proof}
It is possible to choose $z_1'',z_2''\in\Z^{d'}$ and $v''\in\Z^d$, $v''\geq\cvec{0}$ so that $x=Az_1''-v''$, $z_1''\geq z_2''$, $x-\cvec{1}\geq Az_2''$, and furthermore $z_1''\geq z_1,z_1'$ and $z_2''\leq z_2,z_2'$, i.e. $z_1''$ and $z_2''$ are common upper and lower bounds for $z_1,z_1'$ and $z_2,z_2'$. Therefore, to prove the lemma it is sufficient to consider the case of $z_1'\geq z_1$ and $z_2'\leq z_2$.

Observe that
\begin{flalign*}
\lbl(f,(z_2',z_1'))\overset{L.~\ref{pathsum}}&{=}m(n^A,z_2-z_1')\lbl(f,(z_2',z_2)) \\
&+m(n^A,z_1-z_1')\lbl(f,(z_2,z_1))+\lbl(f,(z_1,z_1')) \\
&=m(n,A(z_2-z_1'))\underbrace{\lbl(f,(z_2',z_2))}_{a_2} \\
&+m(n,A(z_1-z_1'))\underbrace{\lbl(f,(z_2,z_1))}_{a_1}+\underbrace{\lbl(f,(z_1,z_1'))}_{a_0}.
\end{flalign*}
By Lemma~\ref{labelbound}
\begin{flalign*}
a_0&<m(n^A,z_1-z_1')=m(n,A(z_1-z_1')) \\
a_1&<m(n^A,z_2-z_1')/m(n^A,z_1-z_1')=m(n,A(z_2-z_1'))/m(n,A(z_1-z_1')),
\end{flalign*}
and therefore
\[\base_{n}(\lbl(f,(z_2',z_1')),(A(z_1-z_1'),A(z_2-z_1')))=(a_0,a_1,a_2).\]
Let
\begin{flalign*}
&\base_{n}(\lbl(f,(z_2,z_1))),(-v,-v-\cvec{1})) \\
&=\base_{n}(a_1,(-v'-A(z_1-z_1'),-v'-A(z_1-z_1')-\cvec{1}))=(b_0,b_1,b_2).
\end{flalign*}
To see that the mixed base expression after this paragraph contains a valid directive sequence, we need to check that $A(z_1-z_1')\geq -v'$ and $-v'-\cvec{1}\geq A(z_2-z_1')$. By definition $Az_1-v=Az_1'-v'$ and thus $-v'=A(z_1-z_1')-v\leq A(z_1-z_1')$. The other inequality is shown by
\[-v'-\cvec{1}=-Az_1'+x-\cvec{1}\geq-Az_1'+Az_2=A(z_2-z_1').\]

By Lemma~\ref{prescope}
\[\base_{n}(\lbl(f,(z_2',z_1')),(A(z_1-z_1'),-v',-v'-\cvec{1},A(z_2-z_1')))=(a_0,b_0,b_1,b_2,a_2).\]
We conclude by computing
\begin{flalign*}
&\base_{n}(a,(-v,-v-\cvec{1}))[2]=\base_{n}(\lbl(f,(z_1,z_2))),(-v,-v-\cvec{1}))[2] \\
&=\base_{n}(\lbl(f,(z_2',z_1')),(A(z_1-z_1'),-v',-v'-\cvec{1},A(z_2-z_1')))[3] \\
\overset{L.~\ref{endscope}}&{=}\base_{n}(\lbl(f,(z_2',z_1')),(-v',-v'-\cvec{1}))[2]=\base_{n}(a',(-v',-v'-\cvec{1}))[2].
\end{flalign*}
\end{proof}

\begin{lemma}\label{microfacesloz}
Let $f\in X_{n^A}$ and $x\in\Z^d$. Let $z_1\in\Z^{d'}$, $v\in\Z^{d}$, $v\geq\cvec{0}$ be such that $x=Az_1-v$. Let $z_2\in\Z^{d'}$ satisfy $z_1\geq z_2$ and $x-\cvec{1}\geq Az_2$. Denote $a=\lbl(f,(z_2,z_1))$. For any face $s=\faceat(p,u)$ it holds that 
\[s(\micro_A(f)[x])=\base_{n}(a,(-v+p,-v+p+u))[2].\]
\end{lemma}
\begin{proof}
Write $\base_{n}(a,(-v,-v-\cvec{1}))=(a_0,a_1,a_2)$. By Lemma~\ref{twostepbase}
\[\base_{n}(a,(-v+p,-v+p+u))[2]=\base_{n}(a_1,(p,p+u))[2],\]
so we can compute
\begin{flalign*}
&\faceat(p,u)(\micro_A(f)[Az_1-v])=\base_{n}(\base_{n}(a,(-v,-v-\cvec{1}))[2],(p,p+u))[2] \\
&=\base_{n}(a_1,(p,p+u))[2]=\base_{n}(a,(-v+p,-v+p+u))[2].
\end{flalign*}
\end{proof}

\begin{theorem}\label{microMatch}
Let $f\in X_{n^A}$. Then $\micro_A(f)\in X_{n}$.
\end{theorem}
\begin{proof}
We need to show that $\topi_i(\micro_A(f)[w])=\boti_i(\micro_A(f)[w+e_i])$ for any $w\in\Z^d$ and $1\leq i\leq d$. Let $z_1,z_2\in\Z^{d'}$ be such that Lemma~\ref{microfacesloz} can be applied at coordinates $w$ and $w+e_i$. Let $w=Az_1-v$ and $w+e_i=Az_1-u$, from which it follows that $-v+e_i=-u$. Denote $a=\lbl(f,(z_2,z_1))$. By applying Lemma~\ref{microfacesloz} two times we see that

\begin{flalign*}
&\topi_i(\micro_A(f)[w])=\base_{n}(a,(-v,-v-\cvec{1}+e_i))[2] \\
&=\base_{n}(a,((-v+e_i)-e_i,(-v+e_i)-e_i+(-\cvec{1}+e_i)))[2] \\
&=\base_{n}(a,(-u-e_i,-u-e_i+(-\cvec{1}+e_i)))[2] \\
&=\boti_i(\micro_A(f)[w+e_i]).
\end{flalign*}
\end{proof}

\begin{lemma}\label{macroedgeloz}
Let $f\in X_{n^A}$ and let $p,p'\in\Z^{d'}$. Then 
\[\lbl(\micro_A(f),(Ap,Ap'))=\lbl(f,(p,p'))\quad\mbox{and}\quad(Ap,Ap')\micro_A(f)=(p,p')f.\]
\end{lemma}
\begin{proof}
The two equalities are equivalent, so it suffices to prove the first one. We first show for $z_1\in\Z^{d'},v\in\Z^d$, $v\geq\cvec{0}$ that
\[\lbl(\micro_A(f),(A(z_1-v),Az))=\lbl(f,(z_1-v,z_1)).\]
We can represent $Av=\sum_{k=1}^m e_{i_k}$ for some $m\in\N$. Denote $v'_j=\sum_{k=1}^j e_{i_k}$ for $0\leq j\leq m$, so in particular $Av=v'_m$. Denote $s_i=\faceat(\cvec{0},-e_i)$ for any $1\leq i\leq m$. Choose $z_2\in\Z^{d'}$ so that Lemma~\ref{microfacesloz} can be applied to faces $s_{i_{j+1}}$ of the microtiles at coordinates of the form $Az_1-v'_j$. Let $a=\lbl(f,(z_2,z_1))$.

\begin{flalign*}
&\lbl(\micro_A(f),(A(z-v),Az))=\lbl(\micro_A(f),(Az-v'_m,Az-v'_0)) \\
\overset{L.~\ref{pathsum}}&{=}\sum_{j=0}^{m-1}m(n,-v'_j)\lbl(\micro_A(f),(Az-v'_{j+1},Az-v'_j)) \\
&=\sum_{j=0}^{m-1}m(n,-v'_j)s_{i_{j+1}}(\micro_A(f)[Az-v'_j]) \\
\overset{L.~\ref{microfacesloz}}&{=}\sum_{j=0}^{m-1}m(n,-v'_j)\base_{n}(a,(-v'_j,\underbrace{-v'_j-e_{i_{j+1}}}_{-v'_{j+1}}))[2] \\
\overset{L.~\ref{prebasispath}}&{=}\base_{n}(a,(-v'_0,-v'_m))[2]=\base_{n}(a,(\cvec{0},-Av))[2] \\
&=\base_{n^A}(\lbl(f,(z_2,z_1)),(\cvec{0},-v))[2]\overset{L.~\ref{baselabel}}{=}\lbl(f,(z_1-v,z_1)).
\end{flalign*}
On the last line we need that $(z_1,z_1-v,z_2)$ is a directive sequence, but this can be achieved by choosing $z_2$ suitably.

For the general case, let $p''\in\Z^d$ such that $p''\geq p$ and $p''\geq p'$. The previous part is applied in the following with the choices $z=p''$ and $v=p''-p$ or $v=p''-p'$.
\begin{flalign*}
&\lbl(\micro_A(f),(Ap,Ap')) \\
\overset{L.~\ref{pathsum}}&{=}m(n,A(p''-p'))\lbl(\micro_A(f),(Ap,Ap''))+\lbl(\micro_A(f),(Ap'',Ap')) \\
&=m(n^A,p''-p')\lbl(f,(p,p''))+\lbl(f,(p'',p'))\overset{L.~\ref{pathsum}}{=}\lbl(f,(p,p')).
\end{flalign*}
\end{proof}

Now we can show that the real number represented by a tessellation is preserved also by the microtile map $\micro_A$. 

\begin{proposition}\label{invtessreal}
Let $A$ be a $d\times d'$ natural number matrix all of whose rows contain a positive integer. For any $f\in X_{n^A}$ it holds that $\real(\micro_A(f))=\real(f)$.
\end{proposition}
\begin{proof}
We observe that $\wgt_n(-\cvec{1})>1$ if and only if $\wgt_n(-A\cvec{1})>1$. The case $\wgt_n(-\cvec{1})=\wgt_n(-A\cvec{1})=1$ is simple because then $\real(\micro_A(f))=\real(f)=0$.

Assume then that $\wgt_n(-\cvec{1})>1$ and $\wgt_n(-A\cvec{1})>1$. Because $-\cvec{1}\ll\cvec{0}$ and $-A\cvec{1}\ll\cvec{0}$, it follows from Lemma~\ref{intCutoff} that $\real(\micro_A(f))$ is finite if and only if $\real_{-A\cvec{1}}(\micro_A(f))$ is finite. If $\real(\micro_A(f))$ and $\real_{-A\cvec{1}}(\micro_A(f))$ are both infinite, then trivially $\real(\micro_A(f))=\real_{-A\cvec{1}}(\micro_A(f))$, and by Proposition~\ref{freepv} $\real(\micro_A(f))=\real_{-A\cvec{1}}(\micro_A(f))$ also holds if they are both finite. Therefore
\begin{flalign*}
&\real(\micro_A(f))=\real_{-A\cvec{1}}(\micro_A(f))=\lim_{i\to\infty}(-iA\cvec{1},iA\cvec{1})\micro_A(f) \\
&\overset{L.~\ref{macroedgeloz}}{=}\lim_{i\to\infty}(-\cvec{i},\cvec{i})f=\real(f).
\end{flalign*}
\end{proof}

\begin{theorem}\label{macromicroInv}
Let $A$ be a natural number matrix all of whose rows contain a positive integer. Then $\macro_{A,n}$ and $\micro_{A,n}$ are inverse maps of each other.
\end{theorem}
\begin{proof}
Let $f\in X_{n}$ and consider $x=Az_1-v$ with $z_1\in\Z^{d'}$, $v\in\Z^{d}$, $v\geq\cvec{0}$. Let $z_2\in\Z^{d'}$ satisfy $z_1\geq z_2$ and $x-\cvec{1}\geq Az_2$. Then
\begin{flalign*}
&\micro_A(\macro_A(f))[Az_1-v]=\cube_{n}(\base_{n}(\lbl(\macro_A(f),(z_2,z_1)),(-v,-v-\cvec{1}))[2]) \\
\overset{L.~\ref{macroedgeSurj}}&{=}\cube_{n}(\base_{n}(\lbl(f,(Az_2,Az_1)),(-v,-v-\cvec{1}))[2]) \\
\overset{L.~\ref{baselabel}}&{=}\cube_{n}(\lbl(f,(Az_1-v-\cvec{1},Az_1-v))) \\
\overset{L.~\ref{diaglabelvalue}}&{=}\cube_{n}(\val_{n}(f[Az_1-v]))=f[Az_1-v].
\end{flalign*}
For the other direction, let $f\in X_{n^A}$ and $z\in\Z^{d'}$. Then
\begin{flalign*}
&\macro_A(\micro_A(f))[z]=\cube_{n^A}(\lbl(\micro_A(f),(A(z-\cvec{1}),Az))) \\
&\overset{L.~\ref{macroedgeloz}}{=}\cube_{n^A}(\lbl(f,(z-\cvec{1},z)))\overset{L.~\ref{diaglabelvalue}}{=}\cube_{n^A}(\val_{n^A}(f[z]))=f[z].
\end{flalign*}
\end{proof}

As a corollary we can show that any bi-infinite sequence of cubes occurs on the main diagonal of precisely one valid tessellation.

\begin{corollary}\label{tessExist}
For every $x\in\digs_N^\Z$ there is a unique $f\in X_n$ such that $\val_n(f[\cvec{i}])=x[i]$ for all $i\in\Z$.
\end{corollary}
\begin{proof}
Let $A$ be the $d\times 1$ matrix with all entries equal to $1$, so $n^A=(N)$ and there is a $g\in X_{n^A}=X_{(N)}=T_{(N)}^\Z$ such that $\val_{(N)}(g[i])=x[i]$ for all $i\in\Z$. For $f\in X_n$ the equality $\macro_A(f)=g$ holds if and only if
\begin{flalign*}
&\val_n(f)[\cvec{i}]=\lbl(f,(\cvec{i-1},\cvec{i}))=\lbl(f,(A(i-1),Ai))\overset{L.~\ref{macroedgeSurj}}{=}\lbl(\macro_A(f),(i-1,i)) \\
&=\val_{(N)}(\macro_A(f)[i])=\val_{(N)}(g[i])=x[i]\mbox{ for all }i\in\Z,
\end{flalign*}
so we need to show that there is a unique $f\in X_n$ such that $\macro_A(f)=g$. This follows because by Theorem~\ref{macromicroInv} the map $\macro_A$ is bijective.
\end{proof}

We will show that $\macro_A$ is surjective for a $d\times d'$ natural number matrix $A$ even without the assumption that all its rows contain a positive integer.

\begin{theorem}\label{macroSurj}
Let $A$ be a $d\times d'$ natural number matrix. Then $\macro_A$ is surjective.
\end{theorem}
\begin{proof}

We have $\macro_A=\macro_{A,n}$ for a prebasis $n$. Let $B$ be a $d'\times 1$ matrix with all entries equal to $1$. By Theorem~\ref{macromicroInv} the map $\macro_{B,n^A}$ is bijective, so $\macro_{A,n}$ is surjective if and only if $\macro_{AB,n}=\macro_{B,n^A}\circ\macro_{A,n}$ is surjective. Since $AB$ is a $d\times 1$ matrix, by rewriting $AB$ as $A$ we may assume without loss of generality that $A$ is a $d\times 1$ matrix. Then $n^A=(N')$ for some $N'\in\Zpos$ and $X_{n^A}=T_{(N')}^\Z$.

The case $N'=1$ is simple because then $X_{n^A}$ contains only one point, so we may assume that $N'>1$. By continuity of $\macro_A$ it is sufficient to show that the image of $\macro_A$ contains a dense subset of $X_{n^A}$. One dense subset is given by $g\in X_{n^A}$ such that $\real(g)=\xi$ for some $\xi\in\R\setminus\Q$, and for any other $g'\in X_{n^A}$ with $\real(g')=\xi$ it follows that $g'=g$, because irrational numbers have a unique representation in any integer base. Thus, for any fixed $\xi\in\R\setminus\Q$ and $g\in X_{n^A}$ with $\real(g)=\xi$ we will present an $f\in X_n$ such that $\macro_{A}(f)=g$. We choose $f\in X_n$ with the property that $\real(f)=\real(g)$. This is possible, because $\real(f)$ depends only on the choice of cubes on the main diagonal of $f$, and these can be chosen freely by Corollary~\ref{tessExist}. By Proposition~\ref{surjtessreal} we have $\real(\macro_{A}(f))=\real(f)=\real(g)=\xi$. But because for any $g'\in X_{n^A}$ the equality $\real(g')=\xi$ implies $g'=g$, it follows that $\macro_A(f)=g$.
\end{proof}

\begin{corollary}\label{macroFact}
Let $A$ be a $d\times d'$ natural number matrix. Then $(X_n,\sigma_{Az})$ has $(X_{n^A},\sigma_z)$ as a factor for every $z\in\Z^{d'}$ via $\macro_A$. If every row of $A$ contains a positive number, this is a conjugacy.
\end{corollary}
\begin{proof}
By Theorem~\ref{macroSurj} the map $\macro_A$ is surjective, so it remains to show that $\sigma_z\circ \macro_A=\macro_A\circ\; \sigma_{Az}$. Let $f\in X_{n}$ and $v\in\Z^{d'}$ be arbitrary. Then
\begin{flalign*}
&\sigma_z(\macro_A(f))[v]=\macro_A(f)[z+v]=\cube_{n^A}(\lbl(f,(A(z+v-\cvec{1}),A(z+v)))) \\
\overset{L.~\ref{shiftlabel}}&{=}\cube_{n^A}(\lbl(\sigma_{Az}(f),(A(v-\cvec{1}),Av)))=\macro_A(\sigma_{Az}(f))[v].
\end{flalign*}
In the case when every row of $A$ contains a positive number, the map $\macro_A$ is bijective by Theorem~\ref{macromicroInv}.
\end{proof}

\section{Multiplication automata}\label{CASect}

\subsection{Preliminaries on multiplication automata}\label{CAPre}

Let $\alp$ be an alphabet. Concatenations of symbols from $\alp$ are called words, and for $k\in\Zpos$, the set of words of length $k$ is denoted by $\alp^k$. Elements of the one-dimensional subshift $\alp^\Z$ are bi-infinite sequences that are called configurations. As usual, the value of $x\in\alp^\Z$ at a coordinate $i\in\Z$ is denoted by $x[i]$. More generally, given an interval $I=[i,j]$ with $i\leq j\in\Z$, $x$ has a finite subword denoted by $x[I]=x[i,j]=x[i]x[i+1]\cdots x[j]\in\alp^{j-i+1}$.

\begin{definition}
Let $\alp$ be an alphabet. We say that a map $F:\alp^\Z\to \alp^\Z$ is a \emph{cellular automaton} if for some $m\leq n\in\N$ there exists a map $f:\alp^{m-n+1}\to \alp$ (a so-called $(m,n)$ local rule for $F$) such that $F(x)[i]=f(x[i+m,i+n])$ for $i\in\Z$. $(\digs^\Z,F)$ is a dynamical system.
\end{definition}

Bijective cellular automata are often called reversible, because their inverse maps are also cellular automata.

For a CA $F:\alp^\Z\to \alp^\Z$, a configuration $x\in X$ and an interval $I=[i,j]$ with $i\leq j\in\Z$, the $I$-trace of $x$ (with respect to $F$) is the sequence $\tr_{F,I}(x)$ over the alphabet $\alp^{j-i+1}$ (i.e. the symbols of the alphabet are words over $\alp$) defined by $\tr_{F,I}(x)[t]=F^t(x)[I]$ for $t\in\N$. If $I=\{i\}$ is the degenerate interval, we may write $\tr_{F,i}(x)$ and if $i=0$, we may write $\tr_F(x)$. If the CA $F$ is clear from the context, we may write $\tr_I(x)$. The $I$-trace subshift of $F$ is $(\trsh_I(F),\sigma)$, where
\[\trsh_I(F)=\tr_{F,I}(\alp^\Z)\subseteq (\alp^{j-i+1})^\N.\]
It is easy to see that $\trsh_{[i,j]}(F)=\trsh_{[i-j,0]}(F)$, and therefore the set of all trace subshifts satisfies
\[\{(\trsh_{[i,j]}(F),\sigma)\mid i\leq j\in\Z\}=\{(\trsh_{[-(k-1),0]},\sigma)\mid k\in\Zpos\}.\]
We call $(\trsh_{[-(k-1),0]},\sigma)$ the width $k$ trace subshift of $F$.

For examining multiplication CA it is convenient to define
\begin{flalign*}
&\num{\digs_N}=\{x\in\digs_N^\Z\mid \exists j\in\Z: x[i]=0\mbox{ for }i<j\} \\
&\fin{\digs_N}=\{x\in\digs_N^\Z\mid x[i]\neq 0\mbox{ for finitely many }i\in\Z\}\subseteq \num{\digs_N}
\end{flalign*}

The elements of $\num{\digs_N}$ are analogous to the usual base-$N$ representations of positive numbers, where there may be infinitely many digits to the right of the decimal point but always a finite number of digits to the left of the decimal point (and then the representation can be extended to a bi-infinite sequence by adding an infinite sequence of zeroes to the left end). The elements of $\fin{\digs_N}$ correspond to numbers whose representations are finite to both directions.

\begin{definition}
Let $N>1$ be an integer. If $\xi\geq0$ is a real number and $\xi=\sum_{i=-\infty}^{\infty}{\xi_i N^{i}}$ is the unique base-$N$ expansion of $\xi$ such that $\xi_i\neq N-1$ for infinitely many $i<0$, we define $\config_N(\xi)\in \num{\digs_N}$ by
\[\config_N(\xi)[i]=\xi_{-i}\]
for all $i\in\Z$. In reverse, for $x\in\digs_N^\Z$ we define
\[\real_N(x)=\sum_{i=-\infty}^{\infty}{x[-i] N^{i}}.\]
Clearly $\real_N(\config_N(\xi))=\xi$ and $\config_N(\real_N(x))=x$ for every $\xi\geq0$ and every $x\in\num{\digs_N}$ such that $x[i]\neq N-1$ for infinitely many $i>0$.

Additionally, in the special case $N=1$ we define $\config_1:\{0\}\to\digs_1^\Z$ and $\real_1:\digs_1^\Z\to\{0\}$ in the only possible way.
\end{definition}

\begin{definition}[Multiplication CA]
Let $N\in\Zpos$. For $p\in\Zpos$ dividing $N$ we define a $(0,1)$ local rule $\mul_{p,N}:\digs_{N}\times \digs_{N}\to \digs_{N}$ for a CA $\Mul_{p,N}$ as follows. Let $q\in\Zpos$ be such that $N=pq$. The numbers $a,b\in \digs_{N}$ are represented as $a=a_1q+a_0$ and $b=b_1q+b_0$, where $a_0,b_0\in \digs_q$ and $a_1,b_1\in \digs_p$: such representations always exist and they are unique. Then
\[\mul_{p,N}(a,b)=\mul_{p,pq}(a_1q+a_0,b_1q+b_0)=a_0p+b_1.\]
\end{definition}

It can be shown that the map $\mul_{p,N}$ performs multiplication by $p$ in base $N$ in the sense that $\Mul_{p,n}(\fin{\digs_N})\subseteq\fin{\digs_N}$, $\Mul_{p,N}(\num{\digs_N})\subseteq\num{\digs_N}$ and $\real_N(\Mul_{p,N}(x))=p\real_N(x))$ for $x\in\num{\digs_N}$. To put it shortly, this is because $\mul_{p,N}$ encodes the usual algorithm for long multiplication by $p$ in base $N$ (for more details, see e.g.~\cite{Kari12a,Kop21trace}). All these maps commute because
\[\Mul_{p_1,N}(\Mul_{p_2,N}(x))=\config_N(p_1p_2\real_N(x)))=\Mul_{p_2,N}(\Mul_{p_1,N}(x))\]
for $p_1,p_2$ dividing $N$ and for $x$ in the dense set $\fin{\digs_N}$. They are reversible, because for $pq=N$
\[\mul_{p,N}(\mul_{q,N}(x))=\config_N(N\real_N(x))=\sigma(x)\]
for $x$ in the dense set $\fin{\digs_N}$ and $\sigma$ is reversible. Whenever $\alpha=p/q$, where $p$ and $q$ are products of factors $p_i$ and $q_i$ of $N$, we may define $\Mul_{\alpha,N}$ as the composition of corresponding $\Mul_{p_i,N}$ and $\Mul_{q_i,N}^{-1}$. The map $\Mul_{\alpha,N}$ does not depend on the choice of the composition into $p_i$ and $q_i$ by arguments similar to those above.

\subsection{The connection between multiplication cube tessellations and multiplication automata}\label{WangCASubSect}

Throughout let $n$ be a prebasis and let $N=\prod_{i=1}^d n[i]$. From any configuration $x\in\digs_N^\Z$ one can form a tessellation $\tess_{n}(x)\in X_{n}$ as follows: for any $z\in\Z^d$ we define $\alpha(z,n)=\prod_{i=1}^d n[i]^{z[i]}$ and define
\[\tess_{n}(x)[z]=\cube_{n}(\Mul_{\alpha(z,n),N}(x)[0])\]
(one still has to verify that this is indeed a valid tiling, which we will do in Lemma~\ref{validTess}). On the other hand, from any tessellation $f\in X_{n}$ one can extract a configuration $\diag_{n}(f)\in\digs_N^\Z$ by looking at the tiles on the main diagonal: let 
\[\diag_{n}(f)[i]=\val_{n}(f[\cvec{i}])\mbox{ for }i\in\Z.\]

\begin{lemma}\label{validTess}
For every $x\in\digs_N^\Z$ it holds that $\tess_{n}(x)\in X_{n}$.
\end{lemma}
\begin{proof}
Let $z\in\Z^d$ and $1\leq i\leq d$ be arbitrary. We need to show that $\topi_i(\tess_{n}(x)[z])=\boti_i(\tess_{n}(x)[z+e_i])$. Let $y=\Mul_{\alpha(z,n),N}(x)$. By writing $y[0]=a_1\prod_{j\neq i}n[j]+a_0$ for $0\leq a_0<\prod_{j\neq i}n[i]$ and $0\leq a_1<n[i]$ we see that $\Mul_{n[i],N}(y)[0]=a_0n[i]+a_1'$ for some $0\leq a_1'<n[i]$. Then
\begin{flalign*}
&\topi_i(\tess_{n}(x)[z])=\topi_i(\cube_{n}(\Mul_{\alpha(z,n),N}(x)[0])) \\
&=\topi_i(\cube_{n}(y[0]))=\topi_i(\cube_{n}(a_1\prod_{j\neq i}n[j]+a_0))=a_0\quad\mbox{and} \\
&\boti_i(\tess_{n}(x)[z+e_i])=\boti_i(\cube_{n}(\Mul_{\alpha(z+e_i,n),N}(x)[0])) \\
&=\boti_i(\cube_{n}(\Mul_{n[i],N}(y)[0]))=\boti_i(\cube_{n}(a_0n[i]+a_1'))=a_0.
\end{flalign*}
\end{proof}

\begin{theorem}\label{tessFactor}
For any $z\in\Z^d$ it holds that $\sigma_z{\circ} \tess_{n}=\tess_{n}{\circ}\;\Mul_{\alpha(z,n),N}$.
\end{theorem}
\begin{proof}
Let $x\in\digs_N^\Z$ and $z'\in\Z^d$ be arbitrary. Then
\begin{flalign*}
&\sigma_z(\tess_{n}(x))[z']=\tess_{n}(x)[z'+z]=\cube_{n}(\Mul_{\alpha(z'+z,n),N}(x)[0]) \\
&=\cube_{n}(\Mul_{\alpha(z',n),N}(\Mul_{\alpha(z,n),N}(x))[0])=\tess_{n}(\Mul_{\alpha(z,n),N}(x))[z'].
\end{flalign*}
\end{proof}

We define ``partial'' shifts on macrotilings by conjugating to a microtiling. Let $A$ be a $d\times d'$ natural number matrix all of whose rows contain a positive integer and let $N'=\prod_{i=1}^{d'} n^A[i]$. For any shift map $\sigma_z:X_{n}\to Z_{n}$ with $z\in\Z^d$ define $\sigma_{A,z}:X_{n^A}\to X_{n^A}$ by $\sigma_{A,z}=\macro_A\circ\;\sigma_z\circ\micro_A$.

\begin{lemma}\label{diagFactor}
For any $z\in\Z^d$ it holds that $\Mul_{\alpha(z,n),N'}{\circ}\diag_{n^A}=\diag_{n^A}{\circ}\;\sigma_{A,z}$. In the special case $A=I$ it holds that $\Mul_{\alpha(z,n),N}{\circ}\diag_{n}=\diag_{n}{\circ}\;\sigma_z$
\end{lemma}
\begin{proof}
First observe that $\Mul_{\alpha(z,n),N'}$ always exists, because every factor of $N$ is a factor of $N'$. It is sufficient to show for any $1\leq k\leq d$ that $\Mul_{n[k],N'}\circ\diag_{n^A}=\diag_{n^A}\circ\;\sigma_{A,e_k}$. After composing by $\macro_A$ on the right, we need to show for all $f\in X_{n}$ and $i\in\Z$ that $\Mul_{n[k],N'}(\diag_{n^A}(\macro_A(f)))[i]=\diag_{n^A}(\macro_A(\sigma_{e_k}(f)))[i]$. We compute
\begin{flalign*}
&\Mul_{n[k],N'}(\diag_{n^A}(\macro_A(f)))[i]=\mul_{n[k],N'}(\diag_{n^A}(\macro_A(f))[i],\diag_{n^A}(\macro_A(f))[i+1]) \\
&=\mul_{n[k],N'}(\underbrace{\lbl(f,(A(\cvec{i-1}),A\cvec{i}))}_{a},\underbrace{\lbl(f,(A\cvec{i},A(\cvec{i+1})))}_{b})
\end{flalign*}
and
\begin{flalign*}
&\diag_{n^A}(\macro_A(\sigma_{e_k}(f)))[i]=\val_{n^A}(\macro_A(\sigma_{e_k}(f)))[\cvec{i}] \\
&=\lbl(\sigma_{e_i}(f),(A(\cvec{i-1}),A\cvec{i}))=\underbrace{\lbl(f,(A(\cvec{i-1})+e_i,A\cvec{i}+e_i))}_{c},
\end{flalign*}
where $a,b,c<N'=m(n,-A\cvec{1})$. We need to show that $\mul_{n[k],N'}(a,b)=c$. After defining $a_0,a_1,b_0,b_1$ by
\begin{flalign*}
&\base(a,(N'/n[k]))=\base_{n}(a,(-A\cvec{1}+e_k,-A\cvec{1}))[0,1] \\
&=\base_{n}(a,((A(\cvec{i-1})+e_k)-A\cvec{i},A(\cvec{i-1})-A\cvec{i}))[0,1]=(a_0,a_1) \\
&\base(b,(N'/n[k]))=\base_{n}(b,(-A\cvec{1}+e_k,-A\cvec{1}))[0,1] \\
&=\base_{n}(b,((A\cvec{i}+e_k)-A(\cvec{i+1}),A\cvec{i}-A(\cvec{i+1})))[0,1]=(b_0,b_1),
\end{flalign*}
that amounts to showing that $c=a_0n[k]+b_1$. By Lemma~\ref{baselabel} we see that
\[a_1=\lbl(f,(A(\cvec{i-1}),A(\cvec{i-1})+e_i))\quad\mbox{ and }\quad b_1=\lbl(f,(A\cvec{i},A\cvec{i}+e_i)).\]
By Lemma~\ref{pathsum} we see that
\begin{flalign*}
&a_1N'+c \\
&=\lbl(f,(A(\cvec{i-1}),A(\cvec{i-1})+e_i))m(n,-A\cvec{1})+\lbl(f,(A(\cvec{i-1})+e_i,A\cvec{i}+e_i)) \\
&=\lbl(f,(A(\cvec{i-1}),A\cvec{i}+e_i)) \\
&=\lbl(f,(A(\cvec{i-1}),A\cvec{i}))m(n,-e_k)+\lbl(f,(A\cvec{i},A\cvec{i}+e_i))=an[k]+b_1.
\end{flalign*}
From this one can solve $c=(a-a_1N'/n[k])n[k]+b_1=a_0n[k]+b_1$, and we are done.
\end{proof}

\begin{theorem}\label{tessdiag}
The maps $\tess_{n}$ and $\diag_{n}$ are inverse maps of each other. For any $z\in\Z^d$ the system $(\digs_{N'}^\Z,\Mul_{\alpha(z,n),N'})$ is conjugate to $(X_{n^A},\sigma_{A,z})$ via $\tess_{n^A}$ and $(\digs_N^\Z,\Mul_{\alpha(z,n),N})$ is conjugate to $(X_n,\sigma_{z})$ via $\tess_{n}$.
\end{theorem}
\begin{proof}
First let $x\in\digs_N^\Z$ and $i\in\Z$ be arbitrary. It holds that 
\[\tess_{n}(x)[\cvec{i}]=\cube_{n}(\Mul_{N^i,N}(x)[0])=\cube_{n}(\sigma^i(x)[0])=\cube_{n}(x[i])\]
and therefore $\diag_{n}(\tess_{n}(x))[i]=\val_{n}(\cube_{n}(x[i]))=x[i]$.

Next let $f\in X_{n}$ and $z\in\Z^d$ be arbitrary and denote $f'=\sigma_z(f)$. It holds that
\begin{flalign*}
&\tess_{n}(\diag_{n}(f))[z]=\sigma_z(\tess_{n}(\diag_{n}(f)))[\cvec{0}] \\
\overset{T. \ref{tessFactor}}&{=}\tess_{n}(\Mul_{\alpha(z,n),N}(\diag_{n}(f)))[\cvec{0}] \overset{L. \ref{diagFactor}}{=}\tess_{n}(\diag_{n}(f'))[\cvec{0}] \\
&=\cube_{n}(\diag_{n}(f')[0])=\cube_{n}(\val_{n}(f'[\cvec{0}]))=f'[\cvec{0}]=f[z].
\end{flalign*}

Since $\tess_{n}$ is the inverse of $\diag_n$, the last claim follows from Lemma~\ref{diagFactor}.
\end{proof}

It is a simple observation that the maps we have defined for mapping between tessellations and configurations are such that the represented real numbers are preserved.
\begin{proposition}\label{confreal}
It holds that $\real_N\circ\diag_n=\real_{\cvec{0},\cvec{1}}$ and $\real_{\cvec{0},\cvec{1}}\circ\tess_n=\real_N$.
\end{proposition}
\begin{proof}
The first equality follows directly from definitions. The second equality follows from the first one because $\tess_{n}$ is the inverse of $\diag_{n}$ by Theorem~\ref{tessdiag}.
\end{proof}

\begin{theorem}\label{mulConjThm}
The system $(\digs_{N'}^\Z,\Mul_{\alpha(z,n),N'})$ is conjugate to $(\digs_N^\Z,\Mul_{\alpha(z,n),N})$ via $\phi=\diag_{n}\circ \micro_{A,n}\circ \tess_{n^A}$. Additionally, $\real_{N}\circ\;\phi=\real_{N'}$.
\end{theorem}
\begin{proof}
The system $(\digs_{N'}^\Z,\Mul_{\alpha(z,n),N'})$ is conjugate to $(X_{n^A},\sigma_{A,z})$ via $\tess_{n^A}$ by Theorem~\ref{tessdiag}. The system $(X_{n^A},\sigma_{A,z})$ is conjugate to $(X_n,\sigma_z)$ via $\micro_A$ by definition. The system $(X_n,\sigma_z)$ is conjugate to $(\digs_N^\Z,\Mul_{\alpha(z,n),N})$ via $\diag_{n}$ by Theorem~\ref{tessdiag}. The last claim follows by Propositions~\ref{invtessreal} and~\ref{confreal}.
\end{proof}

Let us specify one particular type of conjugacy. For any $N>1$, let $N=\prod_{i=1}^d p_i^{k_i}$ be its prime decomposition (with $k_i\in\Zpos$ and with the $p_i$ in rising order) and let $n=(p_1^{k_i},\dots,p_d^{k_d})$. Then, let $m=(p_1,\dots,p_d)$ and $M=\prod_{i=1}^d p_i$. If $A_{N}$ is the diagonal matrix $\diag(k_1,\dots,k_d)$, it is easy to see that $n=m^{A_N}$. Then we define $\conj_{N,M}=\diag_{m}\circ \micro_{A_{N},m}\circ \tess_{n}$.

If $N'>1$ is divisible by the same prime numbers as $N$, we extend the definition to $\conj_{N,N'}=\conj_{N',M}^{-1}\circ\conj_{N,M}$. This agrees with the previous definition in the case $N'=M$, because $\conj_{M,M}$ is the identity map.

\begin{corollary}\label{mulConj}
Assume that $N_1,N_2>1$ are divisible by the same prime numbers. Then $(\digs_{N_1}^\Z,\Mul_{\alpha,N_1})$ is conjugate to $(\digs_{N_2}^\Z,\Mul_{\alpha,N_2})$ via $\conj_{N_1,N_2}$ for any $\alpha>0$ such that these maps are defined. Additionally, $\real_{N_2}\circ\conj_{N_1,N_2}=\real_{N_1}$.
\end{corollary}
\begin{proof}
Let $N=\prod_{i=1}^d p_i$, where $p_i$ is the rising sequence of prime numbers that divide $N_1$ and $N_2$. Then also $\Mul_{\alpha,N}$ is defined. By the previous theorem $(\digs_{N_1}^\Z,\Mul_{\alpha,N_1})$ is conjugate to $(\digs_N^\Z,\Mul_{\alpha,N})$ via $\conj_{N_1,N}$ and $(\digs_N^\Z,\Mul_{\alpha,N})$ is conjugate to $(\digs_{N_2}^\Z,\Mul_{\alpha,N_2})$ via $(\conj_{N_2,N})^{-1}$, so $\Mul_{\alpha,N_1}$ is conjugate to $\Mul_{\alpha,N_2}$ via $\conj_{N_1,N_2}$. The last claim also follows from the previous theorem.
\end{proof}

We highlight that, as it turns out, the conjugacy of the particular form $\conj_{N,N^k}$ works in a very transparent way: cells in a configuration $x\in\digs_N^\Z$ are grouped into blocks of $k$ cells, and the $k$ symbols of $\digs_N$ in each block are interpreted as a base-$N$ representation of a number in $\digs_{N^k}$. We define the block map $\block_{N,k}:\digs_N^k\to\digs_{N^k}$ by $\block_{N,k}(a_{k-1}\cdots a_1a_0)=\sum_{i=0}^{k-1}a_i N^i$ for $N>1$, $k\geq 1$ and $a_i\in\digs_N$.

\begin{theorem}\label{macrocell}
Let $N>1$, $k\geq 1$, $x\in\digs_N^\Z$ and $i\in\Z$. Then $\conj_{N,N^k}(x)[i]=\block_{N,k}(x[ik-(k-1),ik])$.
\end{theorem}
\begin{proof}
As in the definition of $\conj_{N,M}$, let $N=\prod_{i=1}^d p_i^{k_i}$ be the prime decomposition of $N$, and thus $N^k=\prod_{i=1}^d p_i^{kk_i}$. Define vectors $n=(p_1^{k_1},\dots,p_d^{k_d})$, $n'=(p^{kk_1},\dots,p^{kk_d})$ and $m=(p_1,\dots,p_d)$.
Define the diagonal matrix $K=\diag(k,\dots,k)$. We note that $A_{N^k}=A_N K$, so by Lemma~\ref{macroComp} $\macro_{A_{N^k},m}=\macro_{K,n}\circ\macro_{A_{N},m}$. We simplify the defining formula of $\conj_{N,N^k}$ as follows:
\begin{flalign*}
\conj_{N,N^k}=&\conj_{N^k,M}^{-1}\circ\conj_{N,M} \\
=&(\diag_{n'}\circ\macro_{A_{N^k},m}\circ\tess_{m})\circ(\diag_{m}\circ \micro_{A_{N},m}\circ \tess_{n}) \\
=&\diag_{n^K}\circ\macro_{A_{N^k},m}\circ \micro_{A_{N},m}\circ \tess_{n}=\diag_{n^K}\circ\macro_{K,n}\circ \tess_{n}.
\end{flalign*}
We apply this to compute
\begin{flalign*}
&\conj_{N,N^k}(x)[i]=\diag_{n^K}(\macro_{K}(\tess_{n}(x)))[i]=\val_{n^K}(\macro_K(\tess_{n}(x))[\cvec{i}]) \\
&=\lbl(\tess_{n}(x),(K(\cvec{i-1}),K\cvec{i}))=\lbl(\tess_{n}(x),(k(\cvec{i-1}),k\cvec{i})) \\
\overset{L.~\ref{pathsum}}&{=}\sum_{j=0}^{k-1}N^j \lbl(\tess_{n}(x),(k\cvec{i}-\cvec{j}-\cvec{1}),k\cvec{i}-\cvec{j})) \\
&=\sum_{j=0}^{k-1}N^j\val_n(\tess_n(x)[k\cvec{i}-\cvec{j}])=\sum_{j=0}^{k-1}N^j x[ik-j]=\block_{N,k}(x[ik-(k-1),ik]).
\end{flalign*}
\end{proof}

We conclude by presenting some factor maps between multiplication automata. Now $A$ may be a general $d\times d'$ natural number matrix.

\begin{theorem}\label{mulFactThm}
The system $(\digs_N^\Z,\Mul_{\alpha(z,n^A),N})$ has $(\digs_{N'}^\Z,\Mul_{\alpha(z,n^A),N'})$ as a factor via $\phi=\diag_{n^A}\circ \macro_{A,n}\circ \tess_{n}$. Additionally, if $N'>1$ and $x\in\digs_N^\Z$ is such that $\real_N(x)$ is finite, then $\real_{N'}(\phi(x))=\real_N(x)$.
\end{theorem}
\begin{proof}
We note that 
\[\alpha(z,n^A)=m(n^A,-z)=m(n,-Az)=\alpha(Az,n),\]
so $\Mul_{\alpha(z,n^A),N}=\Mul_{\alpha(Az,n),N}$ and it is defined.

The system $(\digs_N^\Z,\Mul_{\alpha(Az,n),N})$ has $(X_n,\sigma_{Az})$ as a factor via $\tess_{n}$ by Theorem~\ref{tessdiag}. The system $(X_n,\sigma_{Az})$ has $(X_{n^A},\sigma_{z})$ as a factor via $\macro_A$ by Corollary~\ref{macroFact}. The system $(X_{n^A},\sigma_{z})$ has $(\digs_{N'}^\Z,\Mul_{\alpha(z,n^A),N'})$ as a factor via $\diag_{n^A}$ by Theorem~\ref{tessdiag}. The last claim follows from Propositions~\ref{surjtessreal} and~\ref{confreal}.
\end{proof}

Let us specify one particular type of factor map. Consider any $M>1$ with a prime decomposition of the form $M=\prod_{i=1}^d p_i$ (with the $p_i$ distinct and in rising order) and let $m=(p_1,\dots,p_d)$. If $M'$ divides $M$, there is a subsequence $(k_i)_{i=1}^{d'}$ of $(1,2,\dots,d)$ such that $M=\prod_{i=1}^{d'} p_{k_i}$. Let $m'=(p_{k_1},\dots,p_{k_{d'}})$. If $A_{M,M'}$ is the $d\times d'$ matrix having $e_{k_i}\in\Z^d$ as the $i$th column for $1\leq i\leq d'$, it is easy to see that $m'=m^{A_{M,M'}}$. Then we define $\fact_{M,M'}=\diag_{m'}\circ \macro_{A_{M,M'},m}\circ \tess_{m}$.

If $N,N'>1$ are divisible by the same prime numbers as $M,M'$ respectively, we extend the definition to $\fact_{N,N'}=\conj_{M',N'}\circ\fact_{M,M'}\circ\conj_{N,M}$. This agrees with the previous definition in the case when $N=M$ and $N'=M'$, because $\conj_{M,M}$ and $\conj_{M',M'}$ are identity maps.

\begin{corollary}\label{mulFact}
Let $N_1,N_2>1$ be such that every prime factor of $N_2$ is a prime factor of $N_1$. Then $(\digs_{N_1}^\Z,\Mul_{\alpha,N_1})$ has $(\digs_{N_2}^\Z,\Mul_{\alpha,N_2})$ as a factor via $\fact_{N_1,N_2}$ for any $\alpha>0$ such that these maps are defined. Additionally, if $x\in\digs_{N_1}^\Z$ is such that $\real_{N_1}(x)$ is finite, then $\real_{N_2}(\fact_{N_1,N_2}(x))=\real_{N_1}(x)$.
\end{corollary}
\begin{proof}
Let $M_1=\prod_{i=1}^{d} p_i$ where $p_i$ are the prime numbers that divide $N_1$ (in rising order), and let $M_2=\prod_{i=1}^{d} q_i$ where $q_i$ are the prime numbers that divide $N_2$. Then $M_2$ divides $M_1$.

The system $(\digs_{N_1}^\Z,\Mul_{\alpha,N_1})$ has $(\digs_{M_1}^\Z,\Mul_{\alpha,M_1})$ as a factor via $\conj_{N_1,M_1}$ by Corollary~\ref{mulConj}. The system $(\digs_{M_1}^\Z,\Mul_{\alpha,M_1})$ has $(\digs_{M_2}^\Z,\Mul_{\alpha,M_2})$ as a factor via $\fact_{M_1,M_2}$ by the previous theorem. The system $(\digs_{M_2}^\Z,\Mul_{\alpha,M_2})$ has $(\digs_{N_2}^\Z,\Mul_{\alpha,N_2})$ as a factor via $\conj_{M_2,N_2}$. Therefore $(\digs_{N_1}^\Z,\Mul_{\alpha,N_1})$ has $(\digs_{N_2},\Mul_{\alpha,N_2})$ as a factor via $\fact_{N_1,N_2}$. The last claim follows from Corollary~\ref{mulConj} and Theorem~\ref{mulFactThm}.
\end{proof}

We conclude this subsection by highlighting an alternative characterization for the maps $\conj_{N_1,N_2}$ and $\fact_{N_1,N_2}$ which may be used to ignore the details of their construction: they are fully determined by the requirement of continuity and the fact that they change representations of real numbers in base $N_1$ to representations of the same real numbers in base $N_2$.
\begin{theorem}
Let $N_1,N_2>1$ be such that every prime factor of $N_2$ is a prime factor of $N_1$. If $\phi:\digs_{N_1}^\Z\to\digs_{N_2}^\Z$ is a continuous map and for any $x\in\digs_{N_1}^\Z$ such that $\real_{N_1}(x)$ is finite it holds that $\real_{N_2}(\phi(x))=\real_{N_1}(x)$, then $\phi=\fact_{N_1,N_2}$. Additionally, if $N_1$ and $N_2$ have the same prime factors, then $\phi=\conj_{N_1,N_2}$, and $\phi$ is injective only in this case. If on the other hand $N_2$ has a prime factor that does not divide $N_1$, then a map $\phi$ satisfying the conditions does not exist.
\end{theorem}
\begin{proof}
The claims that $\phi=\fact_{N_1,N_2}$ or $\phi=\conj_{N_1,N_2}$ follow from Corollaries~\ref{mulConj} and~\ref{mulFact} if we show that the map $\phi$ satisfying the conditions in the statement of the theorem is unique, i.e. if $\phi'$ is another such map then $\phi=\phi'$. Since $\phi$ and $\phi'$ are continuous, it is sufficient to show that $\phi(x)=\phi'(x)$ in a dense subset of $\digs_{N_1}^\Z$, so let $x\in\digs_{N_1}^\Z$ be such that $\real_{N_1}(x)\in\R\setminus\Q$. Then
\begin{flalign*}
&\phi(x)=\config_{N_2}(\real_{N_2}(\phi(x)))=\config_{N_2}(\real_{N_1}(x)) \\
&=\config_{N_2}(\real_{N_2}(\phi'(x)))=\phi'(x).
\end{flalign*}

Next we show that if $N_1$ has some prime factor $p$ not dividing $N_2$, then $\phi$ is not injective. The number $p^{-1}$ has a finite base-$N_1$ representation $0.(N_1/p)$, so $p^{-1}$ has two different representations in base $N_1$. On the other hand, $p^{-1}$ does not have a finite base-$N_2$ representation, because otherwise $N_2^ip^{-1}$ would be an integer for a sufficiently large $i\in\N$ even though $p$ does not divide $N_2$. Therefore $p^{-1}$ has only one representation in base $N_2$. Let $x,y\in\digs_{N_1}^\Z$ be distinct configurations such that $\real_{N_1}(x)=\real_{N_1}(y)=p^{-1}$. Then it also holds that $\real_{N_2}(\phi(x))=\real_{N_2}(\phi(y))=p^{-1}$. Because $p^{-1}$ has only one representation in base $N_2$, from this it follows that $\phi(x)=\phi(y)$ and $\phi$ is not injective.

To show the last claim, assume to the contrary that $N_2$ has a prime factor not dividing $N_1$ and that a map $\phi$ satisfying the conditions exists. For $i\in\N$ let $x_i=\config_{N_1}(N_1^i)$, meaning that $\lim_{i\to\infty}x_i=0^\Z$. For all $i\in\N$ it holds that
\[\real_{N_2}(\phi(x_i))=\real_{N_1}(x_1)=N_1^i\not\equiv0\pmod{N_2},\]
and therefore $\phi(x_i)[0]\neq0$ or $\phi(x_i)[1]\neq0$. But then it follows that
\[\phi\left(\lim_{i\to\infty}x_i\right)=\phi(0^\Z)=0^\Z\neq\lim_{i\to\infty}(\phi(x_i)),\]
contradicting the assumption of continuity of $\phi$.
\end{proof}

\subsection{The regularity status of multiplication automata}\label{regSubSect}

In this subsection we consider the question of which multiplication automata are regular. Regularity is presented in~\cite{Kur97} as one way to classify dynamical systems according to the complexity of their time evolution: non-regular systems are in some sense more complex than regular systems. We will use the notions of sofic subshifts and subshifts of finite type (SFT) (for definitions, see e.g.~\cite{LM95}), but for our purposes it is for the most part sufficient to know that sofic subshifts are precisely the subshift factors of SFTs. From this it also follows that all subshift factors of sofic subshifts are sofic. The following definition was given in~\cite{Kur97}.

\begin{definition}
A dynamical system $(X,T)$ with $X$ a zero-dimensional compact metrizable space is called regular if all its one-sided subshift factors are sofic shifts.
\end{definition}

In Section~3 of~\cite{Kur97} it is also observed that to check the regularity of $(X,F)$ for a CA $F$ it is sufficient to check whether all of its trace subshifts are sofic shifts, i.e. whether the trace subshifts of all widths are sofic. Lower width trace subshifts are factors of higher width trace subshifts, so it is sufficient to check whether trace subshifts of arbitrarily large width are sofic.

We will give in Theorems~\ref{fracNonreg} and~\ref{intReg} a complete characterization of regular multiplication automata based on earlier results~\cite{BM97,Kop21trace}.

\begin{theorem}[Corollary~4.19 of~\cite{Kop21trace}]
The system $(\digs_{pq}^\Z,\Mul_{p/q,pq})$ is not regular for coprime $p,q>2$.
\end{theorem}

\begin{theorem}\label{fracNonreg}
The system $(\digs_N^\Z,\Mul_{p/q,N})$ is not regular for any coprime $p,q>2$ such that this map is defined. 
\end{theorem}
\begin{proof}
Since $\Mul_{p/q,N}$ is defined, all the prime factors of $p$ and $q$ divide $N$, so by Corollary~\ref{mulFact} the system $(\digs_N^\Z,\Mul_{p/q,N})$ has $(\digs_{pq}^Z,\Mul_{p/q,pq})$ as a factor. By the previous theorem $(\digs_{pq}^\Z,\Mul_{p/q,pq})$ has a non-sofic subshift $(X,\sigma)$ as a factor. Therefore $(\digs_N^\Z,\Mul_{p/q,N})$ has the non-sofic subshift $(X,\sigma)$ as a factor and $(\digs_N^\Z,\Mul_{p/q,N})$ is not regular. 
\end{proof}

\begin{theorem}[Section~4 of~\cite{BM97}]
The width $1$ trace subshift of the CA $\Mul_{p,pq}$ for $p,q\geq 1$ is an SFT.
\end{theorem}

\begin{theorem}\label{intReg}
The systems $(\digs_N^\Z,\Mul_{p,N})$ and $(\digs_N^\Z,\Mul_{1/p,N})$ are regular for any $p,N\in\Zpos$ such that these maps are defined.
\end{theorem}
\begin{proof}
The maps $\Mul_{p,N}$ and $\Mul_{1/p,N}$ are inverses of each other, so it is sufficient to show that $(\digs_N^\Z,\Mul_{p,N})$ is regular. For $N=1$ the CA $\Mul_{p,N}$ is the identity map on the set with a single point, so we may assume that $N>1$. Regularity is equivalent to showing that for some $M\geq 1$ the trace subshift of arbitrary width $k\geq M$ is sofic. Let $M$ be such that $p$ divides $N^M$ and fix any $k\geq M$. Let $I=[-(k-1),0]$. We need to show that $(X,\sigma)$ with $X=\trsh_I(F)$ is sofic. Let $\block_{N,k}$ be the map defined before Theorem~\ref{macrocell} and for $z\in((\digs_N)^k)^\N$ denote by $\block_{N,k}(z)$ the coordinatewise application of $\block_{N,k}$ to the symbols of $z$. We need to show that the subshift $(\block_{N,k}(X),\sigma)$ conjugate to $(X,\sigma)$ is sofic. Its configurations $\block_{N,k}(\tr_{\Mul_{p,N},I}(x))$ for $x\in\digs_N^\Z$ are precisely of the form
\begin{flalign*}
&\block_{N,k}(\tr_{\Mul_{p,N},I}(x))[i]=\block_{N,k}(\Mul_{p,N}^i(x)[-(k-1),0]) \\
\overset{T.~\ref{macrocell}}&{=}\conj_{N,N^k}(\Mul_{p,N}^i(x))[0]\overset{C.~\ref{mulConj}}{=}\Mul_{p,N^k}^i(\conj_{N,N^k}(x))[0] \\
&=\Mul_{p,N^k}^i(x')[0]=\tr_{\Mul_{p,N^k}}(x')[i]\mbox{ for }i\in\N
\end{flalign*}
where $x'=\conj_{N,N^k}(x)\in(\digs_{N^k})^\Z$, so $\block_{N,k}(\tr_{\Mul_{p,N},I}(x))\in \trsh_{[0,0]}(\Mul_{p,N^k})$. Since $\conj_{N,N^k}$ is a bijection, in fact $\block_{N,k}(X)=\trsh_{[0,0]}(\Mul_{p,N^k})$. Therefore $(\block_{N,k}(X),\sigma)$ is the trace subshift of width $1$ of the CA $\Mul_{p,N^k}$, and it is sofic by the previous theorem.
\end{proof}

\section*{Acknowledgements}
I thank Tristan Stérin for helpful discussions on representing multiplication by Wang tiles. The work was supported by the Finnish Cultural Foundation (grant number 00220510).

\bibliographystyle{plainurl}
\bibliography{mybib}{}

\end{document}